\documentclass{article}
\usepackage{amsmath}
\usepackage{amssymb}
\usepackage{amscd}
\usepackage{multicol}
\usepackage{color}
\usepackage{amsthm}
\usepackage{mathtools}
\mathtoolsset{showonlyrefs}

\begin{document}
\def\ni{\noindent}
\def\t{\theta}
\def\O{\Omega}
\def\S{\Sigma}
\def\e{\epsilon}
\def\lra{\longrightarrow}
\def\R{{\mathbb R}}
\def\N{{\mathbb N}}
\def\Z{{\mathbb Z}}
\def\RR{{{\mathbb R}}^2}
\def\MS{M^{2{\rm x}2}_s}
\def\N{{\bf N}}
\def\l{\lambda}
\def\LL{${\cal L}$}
\def\E{{\cal E}}
\def\a{{\alpha}}
\def\A{A_{\a}}
\def\ta{\t^{\a}}
\def\rot{{\rm rot}}

\newcommand{\vs}[1]{\vskip #1pt}

\newtheorem{Theorem}{Theorem}[section]
\newtheorem{Definition}[Theorem]{Definition}
\newtheorem{corollary}[Theorem]{Corollary}
\newtheorem{proposition}[Theorem]{Proposition}
\newtheorem{examples}[Theorem]{Esempi}
\newtheorem{example}[Theorem]{Example}
\newtheorem{lemma}[Theorem]{Lemma}
\newtheorem{remark}[Theorem]{Remark}

\catcode`\@=11


   \renewcommand{\theequation}{\thesection.\arabic{equation}}
   \renewcommand{\section}%
   {\setcounter{equation}{0}\@startsection {section}{1}{\z@}{-3.5ex  
plus -1ex
    minus -.2ex}{2.3ex plus .2ex}{\Large\bf}}

\title{\bf Unbounded solutions to a system of coupled asymmetric oscillators at resonance\thanks{Under the auspices of INdAM-GNAMPA, Italy. In particular, the first author acknowledges the support of the GNAMPA Project 2020 ''Problemi ai limiti
per l'equazione della curvatura media prescritta''.}
}
\author{A. Boscaggin, W. Dambrosio and D. Papini}

\date{}
\maketitle


\vs{12}

\begin{quote}
\small
{\bf Abstract.} We deal with the following system of coupled asymmetric oscillators
$$
\left\{\begin{array}{l}
\ddot{x}_1+a_1x_1^+-b_1x^-_1+\phi_1(x_2)=p_1(t) \vspace{7pt}\\
\ddot{x}_2+a_2\,x_2^+-b_2\,x^-_2+\phi_2(x_1)=p_2(t),
\end{array}
\right.
$$
where $\phi_i: \mathbb{R} \to \mathbb{R}$ is locally Lipschitz continuous and bounded, $p_i: \mathbb{R} \to \mathbb{R}$ is continuous and $2\pi$-periodic and the positive real numbers $a_i, b_i$ satisfy
$$
\dfrac{1}{\sqrt{a_i}}+\dfrac{1}{\sqrt{b_i}}=\dfrac{2}{n}, \quad \mbox{ for some } n \in \mathbb{N}.
$$
We define a suitable function $L: \mathbb{T}^2 \to \mathbb{R}^2$, appearing as the higher-dimensional generalization of the well known resonance function used in the scalar setting, and we show how 
unbounded solutions to the system can be constructed whenever $L$ has zeros with a special structure.
The proof relies on a careful investigation of the dynamics of the associated (four-dimensional) Poincar\'e map, in action-angle coordinates.
\medbreak
\noindent
{\bf Keywords.} Systems of ODEs, asymmetric oscillators, unbounded solutions, resonance.

\noindent
{\bf AMS Subject Classification.} 34C11, 34C15.
\end{quote}

\section{Introduction}

In this paper, we investigate the existence of unbounded solutions for a system of coupled asymmetric oscillators of the type
\begin{equation}\label{eq-intro}
\left\{\begin{array}{l}
\ddot{x}_1+a_1x_1^+-b_1x^-_1+\phi_1(x_2)=p_1(t) \vspace{7pt}\\
\ddot{x}_2+a_2\,x_2^+-b_2\,x^-_2+\phi_2(x_1)=p_2(t),
\end{array}
\right.
\end{equation}
where, as usual, $x^{\pm} = \max\{\pm x,0\}$ and, for $i=1,2$, $\phi_i: \mathbb{R} \to \mathbb{R}$ is locally Lipschitz continuous and bounded, $p_i: \mathbb{R} \to \mathbb{R}$ is continuous and $2\pi$-periodic. As for the positive real numbers $a_i, b_i$, we assume that
\begin{equation}\label{eq-reso}
\dfrac{1}{\sqrt{a_i}}+\dfrac{1}{\sqrt{b_i}}=\dfrac{2}{n}, \quad \mbox{ for some } n \in \mathbb{N},
\end{equation}
thus implying that each oscillator is at resonance with respect to the same curve of the Fucik spectrum \cite{Fuc76}.

The study of unbounded solutions for oscillators at resonance is a classical topic in the qualitative theory of ordinary differential equations and we refer to \cite{Maw07} for an excellent survey on this subject. In order to motivate our contribution, the crucial reference to be recalled
here is the seminal paper \cite{AloOrt98} by Alonso and Ortega. It is proved therein (cf. \cite[Theorem 4.1]{AloOrt98}) that, for the scalar asymmetric oscillator
\begin{equation}\label{scalare-intro}
\ddot x + a x^+ - b x^- = p(t), \qquad x \in \mathbb{R},
\end{equation}
with $1/\sqrt{a} + 1/\sqrt{b} = 2/n$, all large solutions are unbounded (either in the past or in the future) whenever the $2\pi$-periodic function
\begin{equation}\label{funzione-risonanza}
\Phi(\theta) = \int_0^{2\pi} C\left( \frac{\theta}{n} + t \right)p(t)\,dt, \qquad \theta \in \mathbb{R}
\end{equation}
has zeros, all simple (in the above formula, $C$ stands for the asymmetric cosine function, cf. Section \ref{subsec-ascos}).
The function $\Phi$, sometimes referred to as resonance function, was previously introduced by Dancer \cite{Dan76} to investigate the $2\pi$-periodic solvability of equation \eqref{scalare-intro}. 
In the linear case ($a= b=n^2$), the function $\Phi$ has (simple) zeros if and only if $\int_0^{2\pi} p(t)e^{-\textrm{i}nt}\,dt \neq 0$: in this case, as well known, all the solutions of $\ddot x + n^2 x = p(t)$ are unbounded; instead, $2\pi$-periodic and unbounded solutions to \eqref{scalare-intro} can coexist in the genuinely asymmetric case  $a \neq b$. The proof of this result was obtained by a careful investigation of the dynamics of the associated Poincar\'e map: more precisely, the zeros of the function $\Phi$ were shown to give rise to invariant sets for the discrete dynamical system associated with \eqref{scalare-intro} and eventually to the existence of unbounded orbits.
Generalization of this approach, requiring the introduction of suitable resonance functions, were later provided for forced asymmetric oscillators
$$
\ddot x + a x^+ - b x^- + \phi(x) = p(t), \qquad x \in \mathbb{R},
$$ 
with $\phi: \mathbb{R} \to \mathbb{R}$ a bounded function (see \cite{Dam02,FabMaw00}) and, more in general, for planar system of the type
$$
Jz' = \nabla H(z) + R(z) + e(t), \qquad z \in \mathbb{R}^2,
$$
where $J$ is the standard symplectic matrix, $H: \mathbb{R}^2 \to \mathbb{R}$ is positive and positively homoegeneous of degree $2$ and
$R: \mathbb{R}^2 \to \mathbb{R}^2$ is bounded (see \cite{FabFon05,Fon04}). We also refer to \cite{AloOrt96,CapDamMaWan13,CapDamWan08,LiuTorQia15,Ma15,MaWan13,Yan04a,Yan04} for related results.

In spite of this extensive bibliography, the existence of unbounded solutions for systems of coupled oscillators seems to be an essentially unexplored topic. To the best of our knowledge, the only available results are the one contained in the recent paper \cite{BosDamPapPP}, dealing however with systems of equations looking like weakly coupled perturbations of linear oscillators
(i.e. $a_i = b_i = n_i^2$ for $i=1,2$) and not being applicable to the more general setting of \eqref{eq-intro}.

The aim of the present paper is to extend the approach of \cite{AloOrt98} in this higher-dimensional framework. 
As expected, this is a quite delicate task, since it leads to the study of the dynamics of a four-dimensional map; nonetheless, we will succeed in providing some partial generalizations of the results in \cite{AloOrt98}. 
In more details, our strategy and results can be described as follows.

In Section \ref{sec-2} we pass to an appropriate set of action-angle coordinates and we perform an asymptotic expansion, at infinity, of the Poincar\'e map associated with \eqref{eq-intro}, cf. \eqref{sistema-asintotico}. In doing this, we are led to define a resonance function defined on the two-dimensional torus,
$$
L: \mathbb{T}^2 \to \mathbb{R}^2, \qquad (\theta_1,\theta_2) \mapsto (L_1(\theta_1,\theta_2),L_2(\theta_1,\theta_2))
$$
which can be thought as the higher-dimensional generalization of the resonance function $\Phi$ defined in \eqref{funzione-risonanza}, see \eqref{proof6}-\eqref{eq-defL}. We notice that when system \eqref{eq-intro} is uncoupled (that is, $\phi_1 = \phi_2 = 0$), then $L(\theta_1,\theta_2) =
(L_1(\theta_1),L_2(\theta_2))$ and, up to a constant, $L_i = \Phi$ with $p = p_i$. 

In Section \ref{sec-3} we investigate the dynamics of this four-dimensional Poincar\'e map and we construct invariant sets, giving rise to unbounded orbits. As in the two-dimensional setting, the zeros of the function
$L$ are shown to play a role; however, due to the coupling terms in system \eqref{eq-intro}, we need here to assume that the Jacobian matrix $JL$ has a special structure at the zeros. More precisely, we introduce the notion of $\mathcal{D}^{\pm}$-matrix, cf. Definition \ref{def-dpiumeno}: again, we observe that such a condition is satisfied by diagonal matrices with concordant sign diagonal entries and, hence, by the matrix $JL$ when system \eqref{eq-intro} is uncoupled and the functions $L_i$ have simple zeros, as in the main result of \cite{AloOrt98}. 
This is a quite technical part of the proof, involving, among other things, a delicate estimate for the $2$-norm of a two-parameter family of suitable matrices, which are perturbations of the identity by $\mathcal{D}^{\pm}$-matrices, cf. Lemma \ref{lem-matrici}.

In Section \ref{sec-4} we finally give our main result for the existence of unbounded solutions to system \eqref{eq-intro}, Theorem \ref{teo-main}.
It provides a positive measure set of initial conditions giving rise to unbounded orbits to \eqref{eq-intro}, whenever the function $L$ has a zero $\omega\in \mathbb{T}^2$ such that the Jacobian matrix $JL(\omega)$ is a  $\mathcal{D}^{\pm}$-matrix.
 Theorem \ref{teo-main}. Notice that this can be interpreted as a kind of local version of the main result in \cite{AloOrt98}. Indeed, we do not claim that every large solution of \eqref{eq-intro} is unbounded: due to the higher-dimensional setting, obtaining this global information seems to be a very hard task, even in the case when all the zeros of $L$ are such that the Jacobian at each zero is a  $\mathcal{D}^{\pm}$-matrix. We mention that the condition for $JL$ to be a $\mathcal{D}^{\pm}$-matrix
can be, in general, not easy to verify. To this end, we discuss some situations in which this can be done and Theorem \ref{teo-main} can thus be applied. The first, quite natural, possibility that we present is a semi-perturbative result (cf. Corollary \ref{cor-piccolo}), dealing with the case in which the $L^{\infty}$-norms of the coupling terms $\phi_1,\phi_2$ are not too big: it is worth noticing that this provided a genuinely asymmetric (non-quantitative) generalization of a result obtained in \cite{BosDamPapPP} for coupled linear oscillators. Other results, more global in nature but focusing on specific choices for the parameters $a_i, b_i$ or the forcing terms $p_i$, are given by Corollary \ref{cor-coeffnullo} and Corollary \ref{cor-noJo}. It seems that various other situations could be treated at the expenses of longer computations.

We finally mention that it should be possible, with the same approach, to consider also the more general case of resonance with respect to different curves of the Fucik spectrum, that is, $1/\sqrt{a_i} + 1/\sqrt{b_i} = 2/n_i$ with $n_i \in \mathbb{N}$. Also, the possibility of coupling more oscillators in a cyclic way $\phi_{i+1} = \phi_i$ could be considered. All these generalizations, however, seem to require substantial technical modifications of the proofs and they are thus postponed to future investigations.

\medbreak
\noindent
\textbf{Notation.} Throughout the paper, the symbol $\Vert \cdot \Vert$ will be used for the Euclidean norm of a vector in the plane. Also, for the index $i=1,2$, we will adopt the cyclic agreement $i+1=2$ for $i=1$.

\section{Coupled asymmetric oscillators: some preliminary estimates}\label{sec-2}

In this section, we perform some preliminary estimates for the solutions of system \eqref{eq-intro}, with the final goal of obtaining an asymptotic expansion for its Poincar\'e-map in action-angle coordinates (see Section \ref{subsec-poincare}).

From now on, as in the Introduction we will always assume that, for $i=1,2$, the positive real numbers $a_i, b_i$ satisfy \eqref{eq-reso}, the function $p_i: \mathbb{R} \to \mathbb{R}$ is continuous and $2\pi$-periodic and the function $\phi_i: \mathbb{R} \to \mathbb{R}$ is locally Lipschitz continuous and bounded. Furthermore, we also suppose that there exist
\begin{equation} \label{eq-limitiphii}
\lim_{x\to \pm \infty} \phi_i(x):=\phi_i(\pm \infty);
\end{equation}
moreover, without loss of generality, 
\begin{equation} \label{eq-rellimitiphii}
\phi_i(-\infty)=-\phi_i(+\infty).
\end{equation}

\subsection{Remarks on the asymmetric cosine and related functions}\label{subsec-ascos}

We collect here some results on various functions related to the asymmetric cosine function $C_i$, $i=1,2$, which is defined as the solution of 
\begin{equation}
\left\{\begin{array}{l}
\ddot{x}_i+a_ix_i^+-b_ix^-_i=0\vspace{7pt}\\
x_i(0)=1,\ \dot{x}_i(0)=0,
\end{array}
\right.
\end{equation}
with $a_i$ and $b_i$ as in \eqref{eq-reso}.
We recall that, for every $i=1,2$, the function $C_i$ is even, $\tau:=2\pi/n$-periodic and its explicit expression in $[-\tau/2,\tau/2]$ is
\begin{equation} \label{eq-defCi}
C_i(t)=\left\{\begin{array}{ll}
\cos \sqrt{a_i}t&\mbox{if } |t|\leq \pi/2\sqrt{a_i} \\
&\\
\displaystyle{-\sqrt{\dfrac{a_i}{b_i}}\sin \sqrt{b_i}\left(|t|-\dfrac{\pi}{2\sqrt{a_i}}\right)}&\mbox{if } \pi/2\sqrt{a_i}\leq |t|\leq \tau/2;
\end{array}
\right.
\end{equation}
For future reference, let us observe that $C_i$, $i=1, 2$, when $a_i\neq b_i$ admits the Fourier series expansion
\begin{equation} \label{eq-serieCi}
C_i(t)=\sum_{h=0}^{+\infty} c_{h,i} \cos hnt,\quad t\in \R,
\end{equation}
where
$$
c_{0,i} = \dfrac{2}{\tau}\, \dfrac{b_i-a_i}{b_i \sqrt{a_i}}
$$
and, for $h \geq 1$,
\begin{equation} \label{eq-coeffCi}
c_{h,i} = 
\left\{
\begin{array}{ll}
\dfrac{4}{\tau}\, \dfrac{b_i-a_i}{b_i-h^2 n^2}\, \dfrac{\sqrt{a_i}}{a_i-h^2 n^2}\, \cos \dfrac{hn\pi}{2\sqrt{a_i}} &  \text{if } a_i \neq h^2 n^2, \,b_i \neq h^2 n^2, \vspace{4pt}\\
\dfrac{1}{2h} & \text{otherwise}
\end{array}
\right.
\end{equation}
(see \cite[Lemma 4.2]{AloOrt98}). 

\smallskip
\noindent
In the next sections we will use the integrals of $C_i$ over the sets $J^\pm_{i+1}$ defined by
\begin{equation} \label{es-insiemij}
\begin{array}{l}
\displaystyle{J^+_{i+1}=\left\{t\in [0,2\pi]:\ C_{i+1}\left(\dfrac{\theta_{i+1}}{n}+t\right)>0\right\},\quad J^-_{i+1}=\left\{t\in [0,2\pi]:\ C_{i+1}\left(\dfrac{\theta_{i+1}}{n}+t\right)<0\right\}}
\end{array}
\end{equation}
where $\theta_{i+1}\in \R$ and we have used the notation $i+1=1$ for $i=2$. It is immediate to observe that the fact that $C_i$ and $C_{i+1}$ are both $2\pi/n$-periodic implies that
\[
\int_{J^{\pm}_{i+1}} C_{i}\left(\dfrac{\theta_{i}}{n}+t\right)\,dt, \quad \theta_i\in \R,
\]
does not change if we replace $[0,2\pi]$ in the definition of $J^{\pm}_{i+1}$ by any interval of lenght $2\pi$. In particular, in the computation of the integral of $C_{i}$ on $J_{i+1}^{+}$, we can replace $J^+_{i+1}$ by the set
\[
\bigcup_{k=0}^{n-1} \left(-\dfrac{\pi}{2\sqrt{a_i}}-\dfrac{\theta_{i+1}}{n}+2k\dfrac{\pi}{n}, \dfrac{\pi}{2\sqrt{a_i}}-\dfrac{\theta_{i+1}}{n}+2k\dfrac{\pi}{n}\right),
\]
thus obtaining that
\begin{equation} \label{eq-intCisuj+}
\int_{J^+_{i+1}} C_{i}\left(\dfrac{\theta_{i}}{n}+t\right)\,dt=  n\Lambda_i (\theta_i-\theta_{i+1}),\quad \forall \ \theta_i,\theta_{i+1}\in \R, \quad i=1,2,
\end{equation}
where  $\Lambda_i:\R \to \R$ is defined by
\begin{equation} \label{eq-deflambda}
\Lambda_i(t)=K_i\left(\dfrac{t}{n}+\dfrac{\pi}{2\sqrt{a_{i+1}}}\right)-K_i\left(\dfrac{t}{n}-\dfrac{\pi}{2\sqrt{a_{i+1}}}\right),\quad \forall \ t\in \R,
\end{equation}
being $K_i$ the primitive of $C_i$ such that $K_i(0)=0$.

\medskip
\noindent
A crucial point in our analysis will be the study ot the resolubility of the equation
\begin{equation} \label{eq-darisolvere}
\Lambda_i(t)=\alpha_i,
\end{equation}
where $\alpha_i$ is given by
\begin{equation} \label{eq-defalphai}
\alpha_i = \dfrac{1}{\sqrt{a_i}}-\dfrac{\sqrt{a_i}}{b_i},\quad i=1,2,
\end{equation}
which is is related to $C_i$ by
\begin{equation} \label{eq-intCi}
\int_0^{2\pi} C_{i}\left(t\right)\,dt=2n\alpha_i,\quad i=1,2.
\end{equation}
In particular, we will be interested in the situation where \eqref{eq-darisolvere} has simple solutions; in order to face this problem, let us first concentrate on the range of the function $\Lambda_i$. Introducing the function $\Sigma_i:\R \to \R$ defined by
\begin{equation} \label{eq-defsigma}
\Sigma_i(t)=n\Lambda'_i(t),\quad \forall \ t\in \R,
\end{equation}
is it possible to prove the following result.
\begin{lemma}\label{lem-propLambda}
	The function $\Lambda_{i}$ given in \eqref{eq-deflambda} is even, $2\pi$-periodic, decreasing in $(0,\pi)$ and increasing in $(-\pi,0)$.
\end{lemma}
\begin{proof} Let us first observe that we have
	\begin{equation} \label{eq-espressionesigma}
	\begin{array}{l}
\Sigma_i(t)=C_i\left(\dfrac{t}{n}+\dfrac{\pi}{2\sqrt{a_{i+1}}}\right)-C_i\left(\dfrac{t}{n}-\dfrac{\pi}{2\sqrt{a_{i+1}}}\right)=\\
\\
= C_i^n\left(t+\dfrac{n\pi}{2\sqrt{a_{i+1}}}\right)-C_i^{n}\left(t-\dfrac{n\pi}{2\sqrt{a_{i+1}}}\right),\quad \forall \ t\in \R,
\end{array}
	\end{equation}
	where $C_i^n:\R \to \R$ is defined by
	\[
	C^n_i(t)=C_i\left(\dfrac{t}{n}\right),\quad \forall \ t\in \R.
	\]
The function $C^n_i$ is continuous, $2\pi$-periodic, even and strictly decreasing in $[0,\pi)$; as a consequence, $\Sigma_i$ is continuous, $2\pi$-periodic and odd. As far as the sign of $\Sigma_i$ is concerned, let us observe that
$\Sigma_{i}(t)=0$ if and only if $ t=k\pi $ for some $k\in\Z$.
Indeed, $\Sigma_{i}(t)=0$ if and only if:
\[
\text{either } t + \frac{n\pi}{2\sqrt{a_{i+1}}} = t - \frac{n\pi}{2\sqrt{a_{i+1}}} +2k\pi \quad\text{or}\quad
t + \frac{n\pi}{2\sqrt{a_{i+1}}} = -t + \frac{n\pi}{2\sqrt{a_{i+1}}} +2k\pi
\quad\text{for some } k\in\Z.
\]
Now, since $a_{i+1}>n^{2}/4$, the first alternative cannot hold, and the second one implies that $t=k\pi$.
Therefore $\Sigma_{i}$ has constant sign in $(0,\pi)$ and a straightforward argument shows that
\[
\Sigma_{i}\left(\pi-\frac{n\pi}{2\sqrt{a_{i+1}}}\right) < 0.
\]
From the above described properties of $\Sigma_i$ we immediately deduce the thesis.
\end{proof}

\smallskip
\noindent
From now on, in order to simplify the notation, let us continue with the case $i=1$; the case $i=2$ is completely analogous.

\noindent
From Lemma \ref{lem-propLambda} we deduce that equation \eqref{eq-darisolvere}, with $i=1$, admits simple solutions if and only if 
\[
\Lambda_1(\pi)< \alpha_1 < \Lambda_1(0); 
\]
in general, the validity of this condition depends on the original pairs $(a_1,b_1)$ and $(a_2,b_2)$. Hence, let us define the resolubility set 
\begin{equation}\label{eq-defR}
\mathcal{R}=\left\{ (a_{1},a_{2})\in \left(\dfrac{n^2}{4},+\infty\right) \times \left(\dfrac{n^2}{4},+\infty\right)\subset \mathbb{R}^2 : \ \Lambda_1(\pi)<\alpha_1<\Lambda_1(0) \right\}.
\end{equation}
The complete description of the open set $\mathcal{R}$ is quite difficult; by means of long computations it is possible to show that the vertical sections $\mathcal{R}\cap \{(a_1^*,a_2):\ a_2>n^2/4\}$, with $a_1^*>n^2/4,\ a_1^*\neq n^2$, are bounded. On the other hand, the study of the horizontal sections $\mathcal{R}\cap \{(a_1,a_2^*):\ a_1>n^2/4\}$, with $a_2^*>n^2/4,\ a_2^*\neq n^2$, is much more complicated. However, the following simple result holds true.
\begin{lemma} \label{lem-risol}
The set $\mathcal{R}$ contains the half-lines $ \{(n^{2},a_{2}): a_{2}>n^{2}/4 \} $ and $\{(a_{1},n^{2}):a_{1}>n^{2}/4\}$. 
\end{lemma}
\begin{proof}
Let us first assume that $a_{1}=n^{2}$ (and, thus, $b_{1}=n^{2}$) and fix $\sqrt{a_{2}} > n/2 $; we then have
\begin{align*}
C_{1}(t)  = \cos(nt), \quad K_{1}(t)  = \frac{1}{n}\sin(nt), \quad \Lambda_{1}(t)  = \frac{2}{n} \sin\frac{\pi n}{2\sqrt{a_{2}}} \cos t, \quad \forall t \in \R.
\end{align*}
A simple computation proves that $\Lambda_{1}(\pi) < 0 < \Lambda_{1}(0) $; noticing that $\alpha_{1}=0$, by \eqref{eq-defalphai}, this shows that $(n^2,a_2)\in \mathcal{R}$.

\medskip
\noindent
On the other hand, if $\alpha_{2}=n^{2}$, we have
\begin{align*}
\Lambda_{1}(0) = 2\int_{0}^{\pi/2n} C_{1}(t)dt,\quad 
\Lambda_{1}(\pi) = 2\int_{\pi/2n}^{\pi/n} C_{1}(t)dt.
\end{align*}
Recalling \eqref{eq-intCi}, we deduce that $\Lambda_{1}(\pi) = 2\alpha_{1} - \Lambda_{1}(0) < \Lambda_{1}(0) $
and, thus, $\alpha_{1} < \Lambda_{1}(0)$. From these relations we also obtain $\Lambda_{1}(\pi) = 2\alpha_{1} - \Lambda_{1}(0)< \alpha_{1}$, proving that $(a_1,n^2)\in \mathcal{R}$.
\end{proof}

\subsection{Asymptotic analysis}\label{subsec-poincare}

We now perform an asymptotic expansion of the Poincar\'e map associated to \eqref{eq-intro}.
We adapt the argument of the proof of \cite[Theorem~4.1]{AloOrt98} to our case: we
write \eqref{eq-intro} as a first order system in $(x_1,x_2,y_1,y_2)=(x_1,x_2,\dot{x}_{1},\dot{x}_{2})$ and use the change of variables
\begin{equation}\label{eq-polarmod}
\left\{\begin{aligned}
x_i & = \gamma_i r_i C_i\left(\dfrac{\theta_{i}}{n}\right)\\
y_i & = \gamma_i r_i S_i\left(\dfrac{\theta_{i}}{n}\right)
\end{aligned}
\right.
\quad\text{with } \gamma_i=\sqrt{2n/a_i},
\end{equation}
where the functions $C_i$ and $S_i$ are defined in Subsection~\ref{subsec-ascos}.

It is straightforward to see that \eqref{eq-intro} is formally equivalent to
\begin{equation} \label{thetaI}
\left\{\begin{aligned}                       
\dot{\theta}_{1} & = n-\dfrac{\gamma_1}{2r_{1}}C_1\left(\dfrac{\theta_{1}}{n}\right)
\left[p_1(t)-\phi_1\left(\gamma_2r_{2}C_2\left(\dfrac{\theta_{2}}{n}\right)\right)\right]
\\
\dot{\theta}_{2} &= n-\dfrac{\gamma_2}{2r_{2}}C_2\left(\dfrac{\theta_{2}}{n}\right) \left[p_2(t)-\phi_2\left(\gamma_1r_{1}C_1\left(\dfrac{\theta_{1}}{n}\right)\right)\right]
\\
\dot{r}_{1} & = \dfrac{\gamma_1}{2n} S_1\left(\dfrac{\theta_{1}}{n}\right)
\left[p_1(t)-\phi_1\left(\gamma_2r_{2}C_2\left(\dfrac{\theta_{2}}{n}\right)\right)\right]
\\
\dot{r}_{2} & = \dfrac{\gamma_2}{2n} S_2\left(\dfrac{\theta_{2}}{n}\right) \left[p_2(t)-\phi_2\left(\gamma_1r_{1}C_1\left(\dfrac{\theta_{1}}{n}\right)\right)\right]
\end{aligned}
\right.
\end{equation}
We denote by $(\theta_1,\theta_2,r_1,r_2)$ the solution of \eqref{thetaI} satisfying $(\theta_1,\theta_2,r_1,r_2)(0)=(\theta_{1,0},\theta_{2,0},r_{1,0},r_{2,0})$ and study the behavior of
$(\theta_1,\theta_2,r_1,r_2)(2\pi)$ as $ \min\{r_{1,0},r_{2,0}\}\to +\infty $.
We also set $ \theta_{0} = (\theta_{1,0},\theta_{2,0}) $, $ r_{0} = (r_{1,0},r_{2,0}) $ and
remark that $ \theta_{0} \in \R^{2} $ and $ r_{1,0}, r_{2,0} > 0 $.

The boundedness of $p_i$ and $\phi_i$ implies that $\dot{r}_{i}$ is uniformly bounded and,
hence, we have
\begin{equation}\label{eq-stimaI}
r_{i}=r_{i,0}+O(1) \qquad\text{and}\qquad
r_{i}^{-1}  = r_{i,0}^{-1}+O(r_{i,0}^{-2}) \qquad \text{as } r_{i,0}\to +\infty,
\end{equation}
where these and all the following estimates hold uniformly w.r.t. $t\in[0,2\pi]$, $\theta_{1,0}$, $\theta_{2,0}$ and $r_{i+1,0}$. 
We deduce that 
\begin{equation}\label{eq-thetapunto}
\begin{aligned}
\dot{\theta}_i & = n-\dfrac{\gamma_i}{2} C_i\left(\dfrac{\theta_{i}}{n}\right) \left(\dfrac{1}{r_{i,0}}+O(r_{i,0}^{-2})\right) \left[p_i(t)- \phi_i\left(\gamma_{i+1} r_{i+1}C_{i+1}\left(\dfrac{\theta_{i+1}}{n}\right)\right)\right] \\
& = n-\dfrac{\gamma_i}{2 r_{i,0}} C_i\left(\dfrac{\theta_{i}}{n}\right)
\left[p_i(t)- \phi_i\left(\gamma_{i+1} r_{i+1}C_{i+1}\left(\dfrac{\theta_{i+1}}{n}\right)\right)\right]+O(r_{i,0}^{-2}),
\end{aligned}
\qquad \text{as } r_{i,0}\to +\infty.
\end{equation}
This relation implies that
\[
\dfrac{\theta_{i}}{n}=\dfrac{\theta_{i,0}}{n}+t+O(r_{i,0}^{-1}),\qquad \text{as } r_{i,0}\to +\infty;
\]
and, thus:
\begin{equation} \label{proof1}
\begin{aligned}
C_i\left(\dfrac{\theta_{i}}{n}\right) & = C_i\left(\dfrac{\theta_{i,0}}{n}+t\right)+O(r_{i,0}^{-1}) \\
S_i\left(\dfrac{\theta_{i}}{n}\right) & = S_i\left(\dfrac{\theta_{i,0}}{n}+t\right)+O(r_{i,0}^{-1})
\end{aligned}
\qquad \text{as } r_{i,0}\to +\infty,
\end{equation}
since $C_{i}$ and $S_{i}$ are smooth enough.
By replacing \eqref{proof1} in the last two equations of \eqref{thetaI} we get
\[
\dot{r}_{i} = 
\dfrac{\gamma_i}{2n} S_i\left(\dfrac{\theta_{i,0}}{n}+t\right)
\left[ p_i(t)-
\phi_i\left(\gamma_{i+1}r_{i+1}C_{i+1}\left(\frac{\theta_{i+1}}{n}\right)\right)\right]
+O(r_{i,0}^{-1})
\qquad \text{as } r_{i,0}\to +\infty.
\]
As a consequence, we infer that
\begin{equation} \label{proof2pre}
\begin{aligned}
r_i(2\pi) & =  r_{i,0}
+\dfrac{\gamma_i}{2n}\int_0^{2\pi}  S_i\left(\dfrac{\theta_{i,0}}{n}+t\right) p_i(t) \,dt \\
& \hphantom{=} -\dfrac{\gamma_i}{2n} \int_0^{2\pi} S_i\left(\dfrac{\theta_{i,0}}{n}+t\right)
\phi_i\left(\gamma_{i+1}r_{i+1}C_{i+1}\left(\frac{\theta_{i+1}}{n}\right)\right)dt
+F_{i,i}(\theta_0,r_0)
\end{aligned}
\end{equation}
where
\[
\lim_{r_{i,0}\to +\infty} F_{i,i}(\theta_0,r_0) = 0 \quad \text{uniformly w.r.t. } \theta_{1}, \theta_{2}, \text{ and } r_{i+1,0}.
\]
Now, we deduce from \eqref{proof1} that
$ C_{i+1}(\theta_{i+1}(t)/n) \to C_{i+1}(\theta_{i+1,0}/n+t) $ uniformly w.r.t. $ t\in[0,2\pi]$, $\theta_{1,0}$, $\theta_{2,0}$ and $r_{i,0}$, as $ r_{i+1,0} \to +\infty$ and, setting
\[
\begin{aligned}
& J^{+}_{i+1,0}=\left\{t\in [0,2\pi]:\ C_{i+1}\left(\frac{\theta_{i+1,0}}{n}+t\right)>0\right\}
\\
& J^{-}_{i+1,0}=\left\{t\in [0,2\pi]:\ C_{i+1}\left(\frac{\theta_{i+1,0}}{n}+t\right)<0\right\},
\end{aligned}
\]
we have that
\[
t\in J^{\pm}_{i+1,0} \implies \lim_{r_{i+1,0}\to+\infty} \phi_i\left(\gamma_{i+1}r_{i+1}C_{i+1}\left(\frac{\theta_{i+1}}{n}\right)\right)
= \phi_{i}(\pm\infty).
\]
where these two limits are not uniform w.r.t. $t\in[0,2\pi]$.
However, using that $ \phi_{i} $ is bounded and $ C_{i+1}(\theta_{i+1}(t)/n) $ converges
uniformly, it is possible to show that:
\[
\lim_{r_{i+1,0}\to +\infty}\int_{J^{\pm}_{i+1,0}} S_i\left(\dfrac{\theta_{i,0}}{n}+t\right)
\left[ \phi_{i}(\pm\infty) - \phi_i\left(\gamma_{i+1}r_{i+1}C_{i+1}\left(\frac{\theta_{i+1}}{n}\right)\right) \right] dt =0
\]
uniformly w.r.t. $\theta_{1,0}$, $\theta_{2,0}$ and $r_{i,0}$.
Therefore, we can write equation \eqref{proof2pre} in the following way:
\begin{equation}\label{proof2}
\begin{aligned}
r_i(2\pi) & = r_{i,0}
+\dfrac{\gamma_i}{2n}\int_0^{2\pi}  S_i\left(\dfrac{\theta_{i,0}}{n}+t\right) p_i(t) \,dt \\
&\hphantom{=} -\dfrac{\gamma_i}{2n} \left(\phi_i(+\infty)
\int_{J^+_{i+1,0}} S_i\left(\dfrac{\theta_{i,0}}{n}+t\right) \,dt 
+\phi_i(-\infty) \int_{J^-_{i+1,0}} S_i\left(\dfrac{\theta_{i,0}}{n}+t\right) \,dt\right)
\\
&\hphantom{=} +F_{i,i}(\theta_0,r_0)+F_{i,i+1}(\theta_0,r_0),
\end{aligned}
\end{equation}
where also
\begin{equation} \label{proof3}
\lim_{r_{i+1,0}\to +\infty} F_{i,i+1}(\theta_0,r_0)=0
\qquad \text{uniformly w.r.t. } \theta_{1,0}, \theta_{2,0} \text{ and } r_{i,0}.
\end{equation}
We now substitute \eqref{eq-stimaI} and \eqref{proof1} in \eqref{eq-thetapunto}, obtaining
\[
\dot{\theta}_i =  n-\dfrac{\gamma_i}{2r_{i,0}} C_i\left(\dfrac{\theta_{i,0}}{n} + t \right) 
\left[p_i(t)- \phi_i\left(\gamma_{i+1} r_{i+1}C_{i+1}\left(\dfrac{\theta_{i+1}}{n}\right)\right)\right] + O(r_{i,0}^{-2}),
\qquad \text{as } r_{i,0}\to +\infty.
\]
Integrating on $[0,2\pi]$ and making similar considerations as done for $ r_{i}(2\pi) $, we deduce that
\begin{equation} \label{proof4}
\begin{aligned}
\theta_i(2\pi) & = \theta_{i,0}+2n\pi -
\dfrac{\gamma_i}{2r_{i,0}} \int_0^{2\pi} C_i\left(\dfrac{\theta_{i,0}}{n} + t \right)p_i(t)\,dt \\
& \hphantom{=} + \dfrac{\gamma_i}{2r_{i,0}} \left(
\phi_i(+\infty) \int_{J^+_{i+1,0}} C_i\left(\dfrac{\theta_{i,0}}{n} + t \right) \,dt 
+\phi_i(-\infty) \int_{J^-_{i+1,0}} C_i\left(\dfrac{\theta_{i,0}}{n} + t \right) \,dt\right)
\\
& \hphantom{=} +\dfrac{1}{r_{i,0}} \left(G_{i,i}(\theta_0,r_0)+G_{i,i+1}(\theta_0,r_0)\right),
\end{aligned}
\end{equation}
where
\begin{equation} \label{proof5}
\lim_{r_{i,0}\to +\infty} G_{i,i}(\theta_0,r_0)
=\lim_{r_{i+1,0}\to +\infty} G_{i,i+1}(\theta_0,r_0)=0
\end{equation}
uniformly w.r.t. the other variables.

\noindent
Recalling \eqref{eq-rellimitiphii}, we observe that \eqref{eq-intCisuj+} and \eqref{eq-intCi} imply that $\theta_{i}(2\pi)$, $i=1,2$, can be written as
\[
\begin{aligned}
\theta_i(2\pi) & = \theta_{i,0}+2n\pi -
\dfrac{\gamma_i}{2r_{i,0}} \int_0^{2\pi} C_i\left(\dfrac{\theta_{i,0}}{n} + t \right)p_i(t)\,dt
 + \frac{1}{r_{i,0}} \gamma_{i}n\phi_{i}(+\infty)(\Lambda_{i}(\theta_{i,0}-\theta_{i+1,0})-\alpha_{i})  
\\
& \hphantom{=} +\dfrac{1}{r_{i,0}} \left(G_{i,i}(\theta_0,r_0)+G_{i,i+1}(\theta_0,r_0)\right),
\end{aligned}
\]
where $\alpha_i, \Lambda_i$ are defined in \eqref{eq-defalphai}, \eqref{eq-deflambda}.

\noindent
For $i=1, 2$, let us now denote
\begin{equation} \label{proof6}
\begin{aligned}
\Phi_i(\theta_{i,0}) & = -\dfrac{\gamma_i}{2}\int_0^{2\pi}  C_i\left(\dfrac{\theta_{i,0}}{n}+t\right)p_i(t)\,dt,
\\
L_i(\theta_{0}) & =\Phi_i(\theta_{i,0})
+\gamma_{i}n\phi_{i}(+\infty)(\Lambda_{i}(\theta_{i,0}-\theta_{i+1,0})-\alpha_{i}),
\end{aligned}
\end{equation}
for every $\theta_0 =(\theta_{1,0},\theta_{2,0} )\in \R^{2} $.
Then, we can summarize \eqref{proof2}, \eqref{proof3}, \eqref{proof4} and \eqref{proof5} as follows:
\begin{equation} \label{sistema-asintotico}
\left\{\begin{aligned}
{\theta_i}(2\pi) & = \theta_{i,0}+2\pi n+\dfrac{1}{r_{i,0}} \left[ L_i(\theta_0)+G_{i,i}(\theta_0,r_0)+G_{i,i+1}(\theta_0,r_0)\right]
\\
r_i(2\pi) & = r_{i,0} -\dfrac{\partial L_i}{\partial \theta_{i,0}} (\theta)+F_{i,i}(\theta_0,r_0)+F_{i,i+1}(\theta_0,r_0)
\end{aligned}
\right.
\qquad \text{for } i=1,2,
\end{equation}
where
\begin{equation} \label{stime-resti}
\begin{aligned}
& \lim_{r_{i,0}\to +\infty} F_{i,i}(\theta_0,r_0)
= \lim_{r_{i,0}\to +\infty} G_{i,i}(\theta_0,r_0) = 0
& &\text{uniformly w.r.t. } r_{i+1,0} \text{ and } \theta_{0}, 
\\
& \lim_{r_{i+1,0}\to +\infty} F_{i,i+1}(\theta_0,r_0)
=\lim_{r_{i+1,0}\to +\infty}  G_{i,i+1}(\theta_0,I_0)=0
& & \text{uniformly w.r.t. } r_{i,0} \text{ and } \theta_{0}.
\end{aligned}
\end{equation}
The functions $L_1, L_2$ will be meant as the components of the vector valued function
\begin{equation}\label{eq-defL}
L(\theta_0) = (L_1(\theta_{0}),L_2(\theta_{0})), \qquad \theta_0 =(\theta_{1,0},\theta_{2,0} )\in \R^{2}.
\end{equation}
which we will call \emph{resonance function} for system \eqref{eq-intro}.
Notice that, due to the $2\pi$-periodicity in both the variables, we can interpret $L$ as a function defined on the two-dimensional torus 
$\mathbb{T}^{2} = \R^{2} / (2\pi\mathbb{Z})^{2}$. This function will play a crucial role in the statement of our main result (see Section \ref{sec-4}).

%
\section{Dynamics of discrete maps}\label{sec-3}

In this section, we establish the abstract result that will be used to prove the existence of unbounded solutions to system \eqref{eq-intro}.
%
\subsection{$\mathcal{D}^\pm$-matrices}
We consider $2\times 2$-matrices $A=(a_{ij})$, $i,j=1, 2$.
\begin{Definition} \label{def-dpiumeno}
	A $2\times 2$-matrix $A$ is said to be a $\mathcal{D}^+$-matrix if 
	\begin{equation} \label{eq-dipiu}
		a_{11}<0,\qquad a_{22}<0, \qquad |a_{12}a_{22}+a_{11}a_{21}|<2a_{11}a_{22}.
	\end{equation}
	Analogously, a $2\times 2$-matrix $A$ is said to be a $\mathcal{D}^-$-matrix if 
	\begin{equation} \label{eq-dimeno}
	a_{11}>0,\qquad a_{22}>0, \qquad |a_{12}a_{22}+a_{11}a_{21}|<2a_{11}a_{22}.
	\end{equation}
\end{Definition}

Notice that a diagonal matrix with negative entries (resp., positive entries) is a $\mathcal{D}^+$ matrix (resp., $\mathcal{D}^-$ matrix). 
Given a $\mathcal{D}^\pm$-matrix $A$ and $\epsilon =(\epsilon_1,\epsilon_2)\in (0,+\infty)^2$, let us define
\begin{equation} \label{eq-matriceb}
B_{\epsilon} =\begin{pmatrix}
1+\epsilon_1 a_{11}&\epsilon_1 a_{12}\\
&\\
\epsilon_2 a_{21}&1+\epsilon_2 a_{22}
\end{pmatrix}.
\end{equation}
Moreover, for every $\epsilon_0>0$ and $\eta>0$ let us define
\begin{equation} \label{eq-defcono}
C_{\epsilon_0,\eta}=\left\{\epsilon=(\epsilon_1,\epsilon_2)\in (0,+\infty)^2:\ \dfrac{a_{11}}{a_{22}}-\eta \leq \dfrac{\epsilon_2}{\epsilon_1}\leq \dfrac{a_{11}}{a_{22}}+\eta ,\quad \|\epsilon\|\leq \epsilon_0\right\}.
\end{equation}
We prove the following result.

\begin{lemma} \label{lem-matrici}
	Assume that $A$ is a $\mathcal{D}^\pm$-matrix. Then, there exist $a_0>0$, $\epsilon_0>0$ and $\eta>0$ such that
	\begin{equation} \label{eq-spettro}
\|B_{\epsilon}\|_{2} \leq 1-\dfrac{1}{2}a_0\|\epsilon\|,\quad \forall \ \epsilon \in C_{\epsilon_0,\eta}.
	\end{equation}
\end{lemma}

\begin{proof} We give the proof in the case of $\mathcal{D}^+$-matrix; the other case is analogous. We recall that the matrix norm $\|B_{\epsilon}\|_{2}$ coincides with the square root of the maximum eigenvalue of the matrix $C_\epsilon=B_\epsilon^T\, B_\epsilon$.
	
	\noindent
	Let us first observe that, for every $\epsilon$, the elements on the diagonal of $C_\epsilon$ are given by
	\[
	(1+\epsilon_1 a_{11})^2+\epsilon_2^2 a_{21}^2,\quad (1+\epsilon_2 a_{22})^2+\epsilon_1^2 a_{12}^2;
	\]
	hence, we have
	\begin{equation} \label{eq-traccia}
	{\rm tr} (C_\epsilon)=2+2(a_{11}\epsilon_1+a_{22}\epsilon_2)+(a_{11}^2+a_{12}^2)\epsilon_1^2+(a_{22}^2+a_{21}^2)\epsilon_2^2.
	\end{equation}
	Hence, a simple computation shows that
	\begin{equation} \label{eq-tracciaquadrato}
	\begin{array}{ll}
	\left({\rm tr} (C_\epsilon)\right)^2=&
	4+8(a_{11}\epsilon_1+a_{22}\epsilon_2)+4(a_{11}\epsilon_1+a_{22}\epsilon_2)^2+4(a_{11}^2+a_{12}^2)\epsilon_1^2+4(a_{22}^2+a_{21}^2)\epsilon_2^2+\\
	&\\
	&+4(a_{11}\epsilon_1+a_{22}\epsilon_2)((a_{11}^2+a_{12}^2)\epsilon_1^2+(a_{22}^2+a_{21}^2)\epsilon_2^2)+((a_{11}^2+a_{12}^2)\epsilon_1^2+(a_{22}^2+a_{21}^2)\epsilon_2^2)^2.
	\end{array}
	\end{equation}
	
	\noindent
	On the other hand, we have
	\begin{equation} \label{eq-determinante}
	\begin{array}{ll}
	{\det} (C_\epsilon)=(\det B_\epsilon)^2=& 1+ 2(a_{11}\epsilon_1+a_{22}\epsilon_2)+(a_{11}\epsilon_1+a_{22}\epsilon_2)^2+2\Delta \epsilon_1 \epsilon_2+\\
	&\\
	&+2(a_{11}\epsilon_1+a_{22}\epsilon_2)\Delta \epsilon_1 \epsilon_2+\Delta^2 \epsilon_1^2 \epsilon_2^2,
	\end{array}
	\end{equation}
	where $\Delta=a_{11}a_{22}-a_{12}a_{21}$. 
	
	\noindent
	Now, let us observe that the matrix $C_\epsilon$ is positive definite; as a consequence, the maximum eigenvalue of $C_\epsilon$ is given by
	\begin{equation} \label{eq-autmax}
	\lambda_+(\epsilon)=\dfrac{{\rm tr}(C_\epsilon)+\sqrt{({\rm tr}(C_\epsilon))^2-4\det (C_\epsilon)}}{2}.
	\end{equation}
	From \eqref{eq-tracciaquadrato} and \eqref{eq-determinante}, by means of simple computations we infer that
	\begin{equation} \label{eq-radice}
	\sqrt{({\rm tr}(C_\epsilon))^2-4\det (C_\epsilon)} =2
\sqrt{ (1+a_{11}\epsilon_1 +a_{22}\epsilon_2)d_2(\epsilon) +P_4(\epsilon)},
	\end{equation}
	where 
	\begin{equation} \label{eq-sottoradice}
	d_2(\epsilon)=(a_{11}^2+a_{12}^2)\epsilon_1^2-2\Delta \epsilon_1 \epsilon_2+(a_{21}^2+a_{22}^2)\epsilon_2^2 
= (a_{11}\epsilon_{1}-a_{22}\epsilon_{2})^{2}+(a_{12}\epsilon_{1}+a_{21}\epsilon_{2})^{2}
	\end{equation}
	and
    \begin{equation} \label{eq-sottoradice2}
	\begin{aligned}
	P_4(\epsilon) & = \dfrac{1}{4}\left[(a_{11}^2+a_{12}^2)\epsilon_1^2+(a_{22}^2+a_{21}^2)\epsilon_2^2\right]^2
                       -\Delta^2 \epsilon_1^2 \epsilon_2^2.
\\
   & =  \dfrac{1}{4} d_{2}(\epsilon)\left[(a_{11}^2+a_{12}^2)\epsilon_1^2+(a_{22}^2+a_{21}^2)\epsilon_2^2
                       +2\Delta \epsilon_1 \epsilon_2\right]
\\
   & = d_{2}(\epsilon) \cdot O(\|\epsilon\|^{2}) \qquad \text{as } \epsilon \to 0.
	\end{aligned}
    \end{equation}
Using \eqref{eq-traccia} and \eqref{eq-autmax}--\eqref{eq-sottoradice2}, we can estimate:
	\begin{equation} \label{eq-autmax3}
    \begin{aligned}
	\lambda_+(\epsilon) & = 1 + a_{11}\epsilon_1+a_{22}\epsilon_2+O(\|\epsilon\|^{2})
                            + \sqrt{d_2(\epsilon)} \, \sqrt{1+a_{11}\epsilon_1 +a_{22}\epsilon_2 +O(\|\epsilon\|^{2})}
\\
    & = 1+ g(\epsilon) + O(\|\epsilon\|^{2}) \qquad \text{as } \epsilon \to 0,
    \end{aligned}
	\end{equation}
	where:
	\[
	g(\epsilon)=a_{11}\epsilon_1+a_{22}\epsilon_2+\sqrt{d_2(\epsilon)}.
	\]
Observe that $g$ is a positively homogeneous function of degree $1$.
A simple computation shows that 
	\begin{equation} \label{eq-pr12}
g\left(-\dfrac{a_{22}}{\sqrt{a_{11}^2+a_{22}^2}},-\dfrac{a_{11}}{\sqrt{a_{11}^2+a_{22}^2}}\right)=\dfrac{-2a_{11}a_{22}+|a_{12}a_{22}+a_{11}a_{21}|}{a_{11}^2+a_{22}^2}:=-4a_0<0,
	\end{equation}
	since the matrix $A$ is a $\mathcal{D}^+$-matrix. Using \eqref{eq-pr12} we deduce that there exists $\eta>0$ such that
	\begin{equation} \label{eq-pr13}
	g\left(\epsilon\right)<-2a_0,
	\end{equation}
	for every $\epsilon$ such that $\|\epsilon\|=1$ and
	\begin{equation} \label{eq-pr15}
	0<\dfrac{a_{11}}{a_{22}}-\eta \leq \dfrac{\epsilon_2}{\epsilon_1}\leq \dfrac{a_{11}}{a_{22}}+\eta.
	\end{equation}
	By homogeneity, we conclude that 	
	\begin{equation} \label{eq-pr14}
	g\left(\epsilon\right)<-2a_0\|\epsilon\|,
	\end{equation}
	for every $\epsilon \in (0,+\infty)^2$ satisfying \eqref{eq-pr15}.

	From \eqref{eq-autmax3} and \eqref{eq-pr14} we deduce that 
there exists $\bar\epsilon>0$ such that 
\begin{equation} \label{eq-autmax5}
\lambda_+(\epsilon)\leq 1-a_0\|\epsilon\|,
\end{equation}
	for every $\epsilon \in C_{{\bar \epsilon},\eta}$. Let us now take $\epsilon_0=\min \{{\bar \epsilon},1/a_0\}$; from \eqref{eq-autmax5} we immediately conclude that
	\[
	\sqrt{\lambda_+(\epsilon)}\leq 1-\dfrac{1}{2}\, a_0 \|\epsilon\|,
	\]
	for every $\epsilon \in C_{{\epsilon}_0,\eta}$.
\end{proof}

\subsection{Invariant sets and unbounded orbits of discrete maps}

In \eqref{sistema-asintotico} we have obtained an estimate for the Poincar\'e map
$ (\theta(0), r(0)) \mapsto (\theta(2\pi), r(2\pi)) $ associated to the system \eqref{thetaI} when
both components $r_{1,0}$ and $r_{2,0}$ of $ r(0) $ are large.
Here we provide sufficient conditions under which the discrete dynamical systems generated by similar maps
possess invariant sets that contain unbounded trajectories.

Few words are in order to clarify the setting in which the dynamical system is defined and represented.
Equations \eqref{sistema-asintotico} define a map $ (\theta,r) \mapsto (u,\rho) $, with
$ \theta = ( \theta_{1}, \theta_{2} ) $,  $r = ( r_{1}, r_{2} ) $, $ u = ( u_{1}, u_{2} ) $ and
$ \rho = ( \rho_{1}, \rho_{2} ) $, such that:
\begin{equation} \label{eq-mappa}
\left\{
\begin{aligned}
u &= \theta+ \left[\begin{array}{c} 2\pi n_1 \\ 2\pi n_2\end{array} \right]
  + \left[ \begin{array}{c} L_1(\theta)/r_{1} \\ L_2(\theta)/r_{2} \end{array} \right]
  + \left[ \begin{array}{c} G_1(\theta,r)/r_{1} \\ G_2(\theta,r)/r_{2} \end{array} \right] \\
\rho & = r - \left[\begin{array}{c} \partial_{1}L_1(\theta) \\ \partial_{2}L_2(\theta)\end{array}\right]
     + F(\theta,r),
\end{aligned}
\right.
\end{equation}
where $ n_{1},n_{2}\in\mathbb{N} $, $ G(\theta,r)=(G_{1}(\theta,r),G_{2}(\theta,r))$ and
$ F(\theta,r)=(F_{1}(\theta,r),F_{2}(\theta,r))$ are continuous,
$ L(\theta)=(L_{1}(\theta),L_{2}(\theta))$ is a $C^{1}$-function with
$ \partial_{j}L_{i} = \partial L_{i}/\partial\theta_{j} $, 
and, moreover, $L,G,F$ are all $2\pi$-periodic w.r.t. $\theta_{1}$ and $\theta_{2}$.
We recall that $ ( \theta_{i}, r_{i} ) $ and $ ( u_{i}, \rho_{i} ) $ are modified polar coordinates 
in $ \R^{2} $ according to \eqref{eq-polarmod} and, hence, there is a couple of well known issues to
take into account.

The first one concerns the singularity of polar coordinates whenever the radius vanishes and will be
easily dealt with since the invariant sets we are going to define will be contained in a region
where $ \min\{ r_{1}, r_{2}\} \ge R > 0 $.

The second issue is that \eqref{eq-mappa} defines a lifting of the actual dynamical
system that, indeed, acts on $ \mathbb{T}^{2} \times \R_{+}^{2} $, where, as usual,
$ \mathbb{T}^{2} = \R^{2} / (2\pi\mathbb{Z})^{2} $ denotes the two-dimensional torus.
More precisely, the coordinates $(\theta,r)$ and $(u,\rho)$ should be projected to 
$\mathbb{T}^{2}\times\R_{+}^{2}$ to determine the correct behavior of the dynamical system, but 
computations are more easily performed on the ``flat'' covering space $\R^{2}\times\R_{+}^{2}$.
To this aim, we denote by $\bar{\theta}_{i}$ the equivalence class of $\theta_{i}$ in
$ \mathbb{T}^{1} = \R/2\pi\mathbb{Z} $ and, thus, we will have
$\bar{\theta} = (\bar{\theta}_{1}, \bar{\theta}_{2}) \in \mathbb{T}^{2}$ for each
$\theta=(\theta_{1},\theta_{2})\in\R^{2}$;
the group metrics in $\mathbb{T}^1$ and $\mathbb{T}^{2} $ are respectively defined by
\begin{equation}\label{eq-metriche}
|\bar \theta_{i} - \bar u_{i} |=\min \{|\theta_{i}-u_{i} + 2n\pi|: n\in \mathbb{Z}\}
\quad \text{and} \quad
\|\bar\theta - \bar u\| = \sqrt{ |\bar\theta_{1}-\bar u_{1}|^{2}+|\bar\theta_{2}-\bar u_{2}|^{2} }.
\end{equation}
It will be clear from the context when $ |\cdot| $ and $ \|\cdot\| $ are meant on either $ \R $ and $ \R^{2} $ or
$\mathbb{T}^{1}$ and $\mathbb{T}^{2}$, respectively.
In particular, we observe that $ |\bar\theta_{i}-\bar u_{i}| = |\theta_{i} - u_{i}| $ if and only if
$|\theta_{i} - u_{i}| \le \pi $.

The invariant sets we obtain are built around a fixed $ \bar\omega \in \mathbb{T}^{2} $ and
depend of four other parameters as follows:
\begin{equation} \label{eq-definsiemeE}
E_{R,\Theta,\lambda,\eta}=\left\{(\bar\theta,r)\in\mathbb{T}^{2}\times\R_{+}^{2} 
: r_1\geq R,\ r_2\geq R,\ \lambda-\eta\leq \dfrac{r_1}{r_2}\leq \lambda +\eta,\ \|\bar\theta-\bar\omega\|\leq \Theta \right\}.
\end{equation}
where $R>0$, $0<\Theta < \pi$, $\lambda>0$ and $ 0 < \eta < \lambda $.
We will denote by $ f :E_{R,\Theta,\lambda,\eta}\to \mathbb{T}^{2}\times\R_{+}^{2}$ the map which has
\eqref{eq-mappa} as a lifting.
We remark that \emph{all different choices of $n_{1},n_{2}\in\mathbb{Z}$ in \eqref{eq-mappa} define
good liftings of the map $f$:} we will use the choice $n_{1}=n_{2}=0$ in the proof of the next result.
\begin{Theorem} \label{teo-astr}
  In the above setting, let us assume that there exists $\omega\in \R^{2}$ such that $ L(\omega) = 0 $
	and suppose that the Jacobian $JL(\omega)$ is a $\mathcal{D}^+$-matrix. 
	Moreover, assume that
	\begin{equation} \label{eq-iporesti}
	\lim_{\substack{r_i\to +\infty \\ i=1,2}} G(\theta,r)=0\quad\text{and}\quad
\lim_{\substack{r_i\to +\infty \\ i=1,2}} F(\theta,r)=0 \qquad \text{uniformly w.r.t. } \theta.
	\end{equation}
	Then, there exist $R>0$, $\Theta\in\left]0,\pi\right[$, $\lambda>0$ and $\eta\in\left]0,\lambda\right[$ such that:
	\begin{equation} \label{eq-invariantepos}
	f(E_{R,\Theta,\lambda,\eta})\subset E_{R,\Theta,\lambda,\eta}.
	\end{equation}
\end{Theorem}

\begin{proof}
We divide the proof into three parts.
	
	\smallskip
	\noindent
	\textit{Part 1.} Choice of the constants $R, \Theta, \lambda$ and $\eta$.
Let 
	\begin{equation} \label{eq-scelta0}
	\lambda= \dfrac{\partial_1 L_1(\omega)}{\partial_2 L_2(\omega)} > 0,
	\end{equation}
	let $\eta,\epsilon_{0} > 0 $ be as in Lemma \ref{lem-matrici} and let $R_0=1/\epsilon_0$.
	Since $JL(\omega)$ is a $\mathcal{D}^+$-matrix we deduce that there exist $\Theta_0\in\left]0,\pi\right[$ and $\gamma_i>0$, $i=1, 2$, such that
	\begin{equation} \label{eq-scelta1}
	\partial_i L_i (\theta)\leq -\gamma_i<0,\quad \text{for } i=1,2 \text{ and }
\forall \theta\in \RR : \|\bar\theta-\bar\omega\|\leq \Theta_0.
	\end{equation}
Moreover, according to assumption \eqref{eq-iporesti}, let $R_1\geq R_0$ such that
	\begin{equation} \label{eq-scelta2} 
	F_i(\theta,r)\geq -\dfrac{\gamma_i}{2},\quad \text{for } i=1,2 \text{ and }
 \forall \ \theta\in \RR,\ r_1\geq R_1,\ r_2\geq R_1.
	\end{equation}
By the continuity of $JL(\theta)$ in $\theta=\omega$, a simple computation shows that there exists $\Theta_1\in\left]0, \Theta_0\right]$ such that 
	\begin{equation} \label{eq-scelta3}
	\begin{aligned}
\dfrac{\partial_1 L_1(\theta)}{\lambda+\eta}-\partial_2 L_2(\theta)
& \geq \dfrac{-\partial_2 L_2(\omega) \eta}{2(\lambda+\eta)}
\\
\partial_2 L_2(\theta)-\dfrac{\partial_1 L_1(\theta)}{\lambda-\eta}
& \geq \dfrac{-\partial_2 L_2(\omega) \eta}{2(\lambda-\eta)}
	\end{aligned}
\qquad\forall\theta:\|\bar\theta-\bar\omega\| \leq \Theta_1.
	\end{equation}
Moreover, using again assumption \eqref{eq-iporesti}, we deduce that there exists $R_2\geq R_1$ such that 
	\begin{equation} \label{eq-scelta4}
	\begin{aligned}
	\left| F_2(\theta,r)-\dfrac{F_1(\theta,r)}{\lambda+\eta}\right| & <
\dfrac{-\partial_2 L_2(\omega) \eta}{2(\lambda+\eta)}\\
	\displaystyle \left| \dfrac{F_1(\theta,r)}{\lambda-\eta}-F_2(\theta,r)\right| & <
\dfrac{-\partial_2 L_2(\omega) \eta}{2(\lambda-\eta)}
	\end{aligned}
\qquad\text{for each }  r_1\geq R_2,\ r_2\geq R_2 \text{ and } \theta\in \R^{2}.
	\end{equation}

Now, let us write
	\begin{equation} \label{eq-pr23}
	L_{i}(\theta)=\langle \nabla L_i(\omega),\theta-\omega \rangle +\alpha_i(\theta)\|\theta-\omega\|,
	\end{equation}
	for $i=1 ,2$ and $\theta \in \RR$, with
	\[
	\lim_{\theta \to \omega} \alpha(\theta)=0,\quad i=1,2.
	\]
with $\alpha(\theta):=(\alpha_{1}(\theta),\alpha_{2}(\theta))$.
	Then, we choose $\Theta\in\left]0,\Theta_1\right]$ such that
		\begin{equation} \label{eq-scelta5}
	\|(\alpha_1 (\theta),\alpha_2(\theta))\|\leq \dfrac{a_0}{4},\qquad \text{if } \|\theta-\omega\|\leq \Theta,
	\end{equation}
	where $a_0$ is given in Lemma \ref{lem-matrici}. 
	
	\noindent
	Let us now  define
	\begin{equation} \label{eq-pr24}
	L^*=\max\{ \|L(\theta)\| : \|\theta-\omega\|\leq \Theta/2\};
	\end{equation}
	according to assumption \eqref{eq-iporesti}, let $R_3\geq R_2$ be such that 
	\begin{equation} \label{eq-scelta6}
	\displaystyle \|G(\theta,r)\|<\min \left\{ L^*, \dfrac{a_0\, \Theta}{8}\, \right\}
	\quad \text{for every } \theta\in \RR \text{ and } r_1, r_2\geq R_3.
	\end{equation}
Finally, let us fix 
	\begin{equation} \label{eq-scelta7}
	R\geq \max \left\{R_3, \dfrac{4L^*}{\Theta}\right\}
	\end{equation}
	 and consider the set $E_{R,\Theta,\lambda,\eta}$ corresponding to the chosen constants.
From now on, we will simply denote this set by $E$.
\medskip

\noindent
	\textit{Part 2.} Invariance of $E$ with respect to the radial components.
Let us fix $(\theta,r)$ such that $(\bar\theta,r)\in E$ and consider $\rho=(\rho_{1},\rho_{2}) $ given by \eqref{eq-mappa}.
From conditions \eqref{eq-scelta1} and \eqref{eq-scelta2} we immediately deduce that
	\begin{equation} \label{eq-raggiopos}
	\rho_i \geq r_i+\dfrac{\gamma_i}{2} >r_i,\quad \text{for }i=1, 2.
	\end{equation}
	On the other hand, we have $ r_1\leq \left(\lambda+\eta \right)r_2 $ and, then, we infer that
\begin{equation} \label{eq-pr25}
\begin{aligned}	
\dfrac{\rho_1}{\rho_2} & =
\dfrac{r_1-\partial_1 L_1(\theta)+F_1(\theta,r)}{r_2-\partial_2 L_2(\theta)+F_2(\theta,r)}
\\
& \leq (\lambda+\eta) \dfrac{r_2-\dfrac{\partial_1 L_1(\theta)}{\lambda+\eta}+\dfrac{F_1(\theta,r)}{\lambda+\eta}}{r_2-\partial_2 L_2(\theta)+F_2(\theta,r)}
\\
& = (\lambda+\eta) \left(1-\dfrac{\dfrac{\partial_1 L_1(\theta)}{\lambda+\eta}-\partial_2 L_2(\theta)+F_2(\theta,R)-\dfrac{F_1(\theta,r)}{\lambda+\eta}}{r_2-\partial_2 L_2(\theta)+F_2(\theta,r)}\right).
\end{aligned}	
\end{equation}
Let us now observe that \eqref{eq-raggiopos} implies that $r_2-\partial_2 L_2(\theta)+F_2(\theta,r)>0$ in $E$;
moreover, from the first relations in \eqref{eq-scelta3} and \eqref{eq-scelta4}, we deduce that
	\begin{equation} \label{eq-pr26}
\dfrac{\partial_1 L_1(\theta)}{\lambda+\eta} - \partial_2 L_2(\theta) + F_2(\theta,R) - \dfrac{F_1(\theta,r)}{\lambda+\eta}>0.
\end{equation}
From \eqref{eq-pr25} and \eqref{eq-pr25} we thus conclude that
	\begin{equation} \label{eq-bordor+}
	\dfrac{\rho_1}{\rho_2}\leq \lambda+\eta.
	\end{equation}
In an analogous way, taking into account the second relations in \eqref{eq-scelta3} and \eqref{eq-scelta4}, it is possible to prove that
	\begin{equation} \label{eq-bordor-}
	\dfrac{\rho_1}{\rho_2}\geq \lambda-\eta.
	\end{equation}
From \eqref{eq-raggiopos}, \eqref{eq-bordor+} and \eqref{eq-bordor-} we deduce the invariance of the set $E$ with respect to the radial components.
\medskip \\
\textit{Part 3.} Invariance with respect to the angular components.
We have to show that, if $(\bar\theta,r)\in E$ then $\|\bar u - \bar\omega\|\le\Theta$, where $u$ is given in \eqref{eq-mappa}.
By the definition of the metric on $\mathbb{T}^{2}$ in \eqref{eq-metriche} and the choice $\Theta<\pi$,
it is enough to work on the covering space and to prove that for a suitable lifting \eqref{eq-mappa} we 
have $ \|u - \omega\|\le\Theta $, with $\theta\in\R^{2}$ such that 
$\|\theta-\omega\|\le\Theta$, where these last two norms are Euclidean in the covering space $\R^{2}$
of $\mathbb{T}^{2}$.
As already announced just before the statement of the theorem, the choice
$ n_{1}=n_{2}=0 $ in \eqref{eq-mappa} will work here.

Let us split the set $E$ into the following two subsets
	\[
E_{1} = \left\{(\bar\theta,r)\in E:\ \|\bar\theta-\bar\omega\| \leq \dfrac{\Theta}{2} \right\},\quad
E_{2} = \left\{(\bar\theta,r)\in E:\ \dfrac{\Theta}{2}\leq \|\bar\theta-\bar\omega\| \leq \Theta \right\}.
	\]
If $(\bar\theta,r)\in E_{1}$, then, using the first equation in \eqref{eq-mappa}, with $n_{1}=n_{2}=0$,
and also \eqref{eq-pr24}, \eqref{eq-scelta6} and \eqref{eq-scelta7}, we deduce that
\begin{equation} \label{eq-invarianzathetapiccolo}
\| u - \omega\| \leq
\| \theta - \omega\|+\frac{1}{R} \|L(\theta)\|+\dfrac{1}{R} \|G(\theta,r)\|
\leq \|\theta -\omega\| + \frac{1}{R} L^* + \frac{1}{R} L^*
\leq \frac{\Theta}{2}+\frac{2L^*}{R}\leq \Theta.
	\end{equation}
On the other hand, if $(\bar\theta,r)\in E_{2}$,
we use \eqref{eq-pr23} and write:
\[
\begin{aligned}
u - \omega & = B (\theta - \omega) +
\left[
\begin{array}{c} \dfrac{\alpha_{1}(\theta)}{r_{1}} + \dfrac{G_{1}(\theta,r)}{r_{1}\|\theta-\omega\|} \vspace{4pt} \\
                 \dfrac{\alpha_{2}(\theta)}{r_{2}} + \dfrac{G_{2}(\theta,r)}{r_{2}\|\theta-\omega\|} \end{array}
\right] \|\theta-\omega\|,
\end{aligned}
\]
where the matrix $B$ is given by
\[
B = \left(
\begin{array}{cc}
1+\partial_{1}L_1(\omega)/r_1 & \partial_{2}L_1(\omega)/r_1 \vspace{4pt}\\
\partial_{1}L_2(\omega)/r_2 & 1 +\partial_{2}L_2(\omega)/r_2 
\end{array}\right)
\]
and has the form \eqref{eq-matriceb} with $ \epsilon=(1/r_{1},1/r_{2}) $.
Using \eqref{eq-scelta5} and \eqref{eq-scelta6} we deduce that
\[
\| u-\omega\| \le \|B\|_{2}\|\theta-\omega\| +
\left(\|\alpha(\theta)\|\|\epsilon\|+\frac{2\|G(\theta,r)\|}{\Theta}\|\epsilon\|\right)\|\theta-\omega\|
\le \left(\|B\|_{2}+\frac{a_{0}}{2}\|\epsilon\|\right)\Theta
\]
Now, $(\bar\theta,r)\in E$ implies that $\epsilon=(1/r_1,1/r_2)\in C_{\epsilon_0,\eta}$, see \eqref{eq-defcono},
and we can
use Lemma~\ref{lem-matrici} to obtain that $\|B\|_{2}\le(1-a_{0}\|\epsilon\|/2)$ and conclude that
$ \| u-\omega\| \le \Theta $.
\end{proof}
Now, let $(\theta_0,r_0)\in E_{R,\Theta,\lambda,\eta}$, with $E_{R,\Theta,\lambda,\eta}$ given by
Theorem~\ref{teo-astr};
since $E_{R,\Theta,\lambda,\eta}$ is positively invariant, we can recursively define
\[
(\theta_{n+1},r_{n+1})=f(\theta_n,r_n) \in E_{R,\Theta,\lambda,\eta},\quad \forall \ n\geq 0.
\]
From \eqref{eq-raggiopos} we know that
\[
(r_1)_i\geq (r_0)_i+\dfrac{\gamma_i}{2} ,\quad i=1,2,
\]
and iterating we infer that
\[
(r_n)_i\geq (r_0)_i+ n\dfrac{\gamma_i}{2},\quad i=1,2,\quad n\geq 1.
\]
This relation is sufficient to prove the final result of this section.
\begin{Theorem} \label{teo-astr2}
In the same setting of Theorem \ref{teo-astr}, for every $	(\theta_0,r_0)\in E_{R,\Theta,\lambda,\eta}$ we have
	\[
	\lim_{n\to +\infty} (r_n)_i=+\infty, \quad i=1, 2,
	\] 
	where $(\theta_{n+1},r_{n+1})=f(\theta_n,r_n)$, for every $n\geq 0$.
\end{Theorem}

\begin{remark} \label{rem-invarianteneg2}
We observe that, in the case of a one-to-one map $f$ as above, an analogous result can be proved when $JL(\omega)$ is a $\mathcal{D}^-$-matrix; indeed, in this situation there exist $R>0$, $0<\Theta < \pi$, $\lambda>0$ and $ 0 < \eta < \lambda $ such that:
	\[
	f^{-1}(E_{R,\Theta,\lambda,\eta})\subset E_{R,\Theta,\lambda,\eta},
	\]
	Then, for every $(\theta_0,r_0)\in E_{R,\Theta,\lambda,\eta}$ it is possible to define
	\[
	(\theta_{n-1},r_{n-1})=f(\theta_{n},r_{n}),
	\]
	for every $n\leq 0$, and we have 
	\[
	\lim_{n\to -\infty} (r_n)_i=+\infty, \quad i=1, 2.
	\] 	
\end{remark}

\section{The main result and some corollaries}\label{sec-4}

In this section we apply the theory developed in Section \ref{sec-3} in order to prove our main result, dealing with the existence of unbounded solutions to the system
\begin{equation}\label{oscacc}
\left\{\begin{array}{l}
\ddot{x}_1+a_1x_1^+-b_1x^-_1+\phi_1(x_2)=p_1(t) \vspace{7pt}\\
\ddot{x}_2+a_2\,x_2^+-b_2\,x^-_2+\phi_2(x_1)=p_2(t).
\end{array}
\right.
\end{equation}
We recall that, for $i=1,2$, we are assuming the resonance condition
\begin{equation}\label{eq-relaibi}
\dfrac{1}{\sqrt{a_i}}+\dfrac{1}{\sqrt{b_i}}=\dfrac{2}{n}, \quad \mbox{ for some } n \in \mathbb{N},
\end{equation}
Moreover, the function $p_i: \mathbb{R} \to \mathbb{R}$ is continuous and $2\pi$-periodic and the function $\phi_i: \mathbb{R} \to \mathbb{R}$ is locally Lipschitz continuous and bounded, with
\begin{equation} \label{eq-rellimitiphii2}
\phi_i(-\infty)=-\phi_i(+\infty).
\end{equation}
In this setting, and recalling the definition of the function $L$ given in \eqref{proof6}-\eqref{eq-defL}, the following result holds true.

\begin{Theorem} \label{teo-main}
Assume conditions \eqref{eq-relaibi} and \eqref{eq-rellimitiphii2}; moreover, suppose that there exists $\omega \in \mathbb{R}^2$ such that $L(\omega) = 0$ and $JL(\omega)$ is a $\mathcal{D}^+$-matrix. 
Then, there exists an infinite measure set $E\subset \RR\times \RR$ such that 
\begin{equation} \label{eq-tesi1}
\lim_{t\to +\infty} (|x_i(t)|^2+|x'_i(t)|^2) =+\infty, \qquad i=1,2,
\end{equation}
for every solution $x$ of \eqref{oscacc} such that $(x(0),x'(0))\in E$.
\end{Theorem}

\begin{remark} \label{rem-solmenoinfinito}
	According to Remark \ref{rem-invarianteneg2}, an analogous result for $t\to -\infty$ can be proved when $JL(\omega)$ is a $\mathcal{D}^-$-matrix.
\end{remark}

\begin{proof} The result follows from an application of Theorem \ref{teo-astr2}, taking into account the fact that, from \eqref{sistema-asintotico}, the Poincar\'e map associated with \eqref{oscacc} is of the form \eqref{eq-mappa}, with \eqref{stime-resti} implying \eqref{eq-iporesti}.
	
	\noindent
	More precisely, let $E\subset \RR\times \RR$ be the set corresponding, via action-angle coordinates, to the set $E_{R,\Theta,\omega,\lambda,\eta}$ given in the statement of Theorem \ref{teo-astr2} and let $x$ be a solution of \eqref{oscacc} such that $(x(0),x'(0))\in E$. Then, from Theorem \ref{teo-astr2} we infer that
	\[
	\lim_{k\to +\infty} (|x_i(2k\pi)|^2+|x'_i(2k\pi)|^2) =+\infty.
	\]
	The thesis \eqref{eq-tesi1} follows from this relation and an application of Gronwall's lemma (see e.g. \cite[Proof of Th. 41]{AloOrt98}), taking into account the boundedness of $\phi_i$, for $i=1,2$.
\end{proof}

In the rest of the section, we discuss some concrete situations in which the abstract condition on the zeros of the function $L$ is verified, thus providing more explicit corollaries of Theorem \ref{teo-main}, depending on the structure of the set of zeroes of the functions $\Phi_i$, $i=1,2$, defined in \eqref{proof6}. 

The first situation we deal with is the one in which both $\Phi_1$ and $\Phi_2$ have a simple zero (in the scalar setting, this situation was the one treated by \cite[Th. 4.1]{AloOrt98}). More precisely, we assume that there exists $\omega^* = (\omega^*_1,\omega^*_2) \in \mathbb{R}^2$ such that
\begin{equation} \label{eq-zerisemplici}
\Phi_1(\omega^*_1)=\Phi_2(\omega^*_2)=0, \qquad \Phi'_i(\omega^*_i) < 0, \quad i=1,2.
\end{equation}
Under this assumption, the following result holds true.

\begin{corollary} \label{cor-piccolo}
Assume conditions \eqref{eq-relaibi}, \eqref{eq-rellimitiphii2} and \eqref{eq-zerisemplici}. Then, there exists $\phi^*=\phi^*(a_1,a_2,p_1,p_2)>0$ such that, for every functions $\phi_i$ with $|\phi_i(+\infty)|<\phi^*$ ($i=1,2$), there exists an infinite measure set $E\subset \RR\times \RR$ such that
\[
\lim_{t\to +\infty} (|x_i(t)|^2+|x'_i(t)|^2) =+\infty, \qquad i=1,2,
\]
for every solution $x$ of \eqref{oscacc} such that $(x(0),x'(0))\in E$.
\end{corollary}

\begin{proof} Let us observe that, in view of Theorem \ref{teo-main} it is sufficient to prove that, under the given assumptions, there exist $\omega\in \RR$ such that $L(\omega)=0$ and $JL(\omega)$ is a $\mathcal{D}^+$-matrix.
	
	\noindent
	Let us first recall, from \eqref{proof6}, that we have
	\begin{equation} \label{eq-proggi1}
	L_i(\theta)=\Phi_i(\theta_i)+\gamma_i n\phi_i(+\infty) (\Lambda_i (\theta_{i}-\theta_{i+1})-\alpha_i),\quad \forall \ \theta \in \RR,
	\end{equation}
	where $\Lambda_i$ is defined in \eqref{eq-deflambda}. Let us define $H:\RR\times \RR \to \RR$ by
	\begin{equation}
	H(\theta,v)=(\Phi_1(\theta_1)+\gamma_1 n v_1 (\Lambda_1 (\theta_{1}-\theta_{2})-\alpha_1), \Phi_2(\theta_2)+\gamma_2 n v_2 (\Lambda_2 (\theta_{2}-\theta_{1})-\alpha_2)),\quad \forall \ \theta \in \RR,\ v\in \RR.
	\end{equation}
	From \eqref{eq-zerisemplici} we immediately infer that
	\[
	H(\omega^*_1,\omega^*_2,0,0)=0
	\]
	and
	\[
	J_\theta H(\omega^*_1,\omega^*_2,0,0)=\left(\begin{array}{cc}
	\Phi'_1(\omega^*_1)&0\\
	&\\
	0&	\Phi'_2(\omega^*_2)
	\end{array}  
	\right) \neq 0.
	\]
	Hence, by the implicit function theorem, we deduce that there exists $\hat{\phi}>0$ such that for every $(\phi_1(+\infty),\phi_2(+\infty))\in \RR$ with $|\phi_i(+\infty)|<\hat{\phi}$, $i=1,2$, there exists $\omega=\omega(\phi_1(+\infty),\phi_2(+\infty)) \in \mathbb{R}^2$ near $\omega^*$ such that
	\[
	L(\omega)=0.
	\]
	Now, let us observe that
	\[
	J L(\omega)=\left(\begin{array}{cc}
	\Phi'_1(\omega_1)+ \gamma_1  \phi_1(+\infty) \Sigma_1(\omega_1 -\omega_2)  & - \gamma_1  \phi_1(+\infty) \Sigma_1(\omega_1-\omega_2)  \\
	&\\
	- \gamma_2  \phi_2(+\infty) \Sigma_2(\omega_2-\omega_1)     &	\Phi'_2(\omega_2)  + \gamma_2  \phi_2(+\infty) \Sigma_2(\omega_2-\omega_1)
	\end{array}  
	\right),
	\]
	where $\Sigma_i$ is given in \eqref{eq-defsigma}.
	The continuity of $\omega$ as function of $(\phi_1(+\infty),\phi_2(+\infty))$, ensured by the implicit function theorem, implies that
	\[
	\lim_{|(\phi_1(+\infty),\phi_2(+\infty))|\to 0^+} J L(\omega)=\left(\begin{array}{cc}
	\Phi'_1(\omega^*_1)&0\\
	&\\
	0&	\Phi'_2(\omega^*_2)
	\end{array}  
	\right);
	\]
	by \eqref{eq-zerisemplici} the limit matrix is a $\mathcal{D}^+$-matrix. As a consequence, there exists $\phi^* \in (0,\hat{\phi})$ such that for every $(\phi_1(+\infty),\phi_2(+\infty))\in \RR$ with $|\phi_i(+\infty)|<\phi^*$ the matrix $J L(\omega)$ is a $\mathcal{D}^+$-matrix, as well. The result is then proved.
\end{proof}

\begin{remark}
A dual result, ensuring the existence of solutions unbounded in the past, could be proved when \eqref{eq-zerisemplici}
is replaced by
$$
\Phi_1(\omega^*_1)=\Phi_2(\omega^*_2)=0, \qquad \Phi'_i(\omega^*_i) > 0, \quad i=1,2.
$$
We omit the details for briefness.
\end{remark}

\begin{remark} \label{rem-lineare}
	Let us analyze the result of Corollary \ref{cor-piccolo} in the symmetric linear case $a_i=b_i=n^2$, $i=1, 2$. In this situation, in the recent paper \cite{BosDamPapPP} the existence of unbounded solutions has been proved under the assumption 
	\begin{equation} \label{eq-ipoJo}
	4|\phi_i(+\infty)| < |{\widehat p}_{i,n}|^2,\quad i=1,2,
	\end{equation}
	where 
	\begin{equation} \label{eq-fourier}
	{\widehat p}_{i,n}=\int_{0}^{2\pi} p_i(t)e^{int}\,dt
	\end{equation}
	(see Theorem 3.1 in \cite{BosDamPapPP}). The assumption $|\phi_i(+\infty)|<\phi^*$ ($i=1,2$), with $\phi^*=\phi^*(a_1,b_1,p_1,p_2)$, in Corollary \ref{cor-piccolo} is then on the same spirit of \eqref{eq-ipoJo}.
\end{remark}

Let us now focus on the situation where the function $\Phi_1$ (or $\Phi_2$) is identically zero, i.e.
\begin{equation} \label{eq-coeffFnullo}
\Phi_1(\theta_1)=0,\quad \forall \ \theta_1\in \R.
\end{equation}
Incidentally, let us observe that in the linear symmetric case $a_1=b_1=n^2$ assumption \eqref{eq-coeffFnullo} corresponds to the case when the number ${\widehat p}_{1,n}$ in \eqref{eq-fourier} is zero. Instead, in the asymmetric case $a_1\neq b_1$, condition \eqref{eq-coeffFnullo} is more tricky to be checked. However, some examples in which it holds can be provided.
For instance, if $a_{1}$ satisfies
\begin{equation}\label{eq-radicedia}
\frac{\sqrt{a_{1}}}{n}=\frac{s}{1+2k} \qquad \text{for some } s,k\in\mathbb{N} \text{ and } s>k,
\end{equation}
then the Fourier coefficient $c_{s,1}$ of $C_{1}$ vanishes (see \eqref{eq-coeffCi}), and \eqref{eq-coeffFnullo} holds when  $p_1(t)=\cos snt$.

For the sake of brevity and clarity, we present here just a couple of corollaries in which \eqref{eq-coeffFnullo} is assumed.
In the first we suppose that $a_{2}$ is such that
\begin{equation}\label{eq-cr2}
c_{r,2}\neq 0 \quad \text{for some } r\in \mathbb{N},
\end{equation}
and that
\begin{equation} \label{eq-sceltap2}
p_2(t)=\mu \cos rnt,\quad \forall \ t\in \R,
\end{equation}
with $\mu>0$.

\begin{corollary} \label{cor-coeffnullo}
	Let $a_i, b_i > 0$ satisfy, for $i=1,2$, assumption \eqref{eq-relaibi}; moreover, suppose that
	\begin{equation} \label{eq-immaginelambda1}
	(a_{1},a_{2})\in \mathcal{R},
	\end{equation}
where $\mathcal{R}$ is defined in \eqref{eq-defR}, and
that \eqref{eq-cr2} is fulfilled.
	Finally, assume that conditions \eqref{eq-rellimitiphii2}, \eqref{eq-coeffFnullo} and \eqref{eq-sceltap2} are satisfied.
	Then, for every $\phi_1(+\infty)\neq 0$ and for every $\phi_2(+\infty)\in \R$ there exists $\mu^*>0$ such that for every $\mu >\mu^*$ there exist two infinite measure sets $E^\pm\subset \RR\times \RR$ such that:
	\begin{itemize}
	\item for every solution $x$ of \eqref{oscacc} such that $(x(0),x'(0))\in E^+$, 
	\[
	\lim_{t\to +\infty} (|x_i(t)|^2+|x'_i(t)|^2) =+\infty, \qquad i=1,2,
	\]
	\item for every solution $x$ of \eqref{oscacc} such that $(x(0),x'(0))\in E^-$
	\[
	\lim_{t\to -\infty} (|x_i(t)|^2+|x'_i(t)|^2) =+\infty, \qquad i=1,2.
	\] 
	\end{itemize}
\end{corollary}

We observe that it is possible to find situations in which Corollary~\ref{cor-coeffnullo} applies.
Indeed, let us first notice that Lemma~\ref{lem-risol} implies that \eqref{eq-immaginelambda1} holds if $(a_{1},a_{2})$ is close to $(n^{2},n^{2})$.
This happens, for instance if $a_{1}$ satisfies \eqref{eq-radicedia} with $ s=2k $ and $k$ large enough, and if $\sqrt{a_{2}}$ is irrational and close to $n$.
With these choices \eqref{eq-coeffFnullo} holds with
$p_{1}(t)=\cos(2knt)$, while \eqref{eq-cr2} is trivially satisfied (see \eqref{eq-coeffCi}).

\begin{proof} Let us first notice that, from \eqref{proof6} and \eqref{eq-sceltap2}, recalling the Fourier expansion of $C_2$ given in \eqref{eq-serieCi}, we obtain
	\begin{equation} \label{eq-cosenoFnullo}
	\Phi_2(\theta_2)=-\dfrac{\gamma_2}{2} \pi \mu c_{r,2} \cos r \theta_2,\quad \forall \ \theta_2\in \R.
	\end{equation}
	As a consequence, recalling  \eqref{eq-coeffFnullo}, we obtain
	\begin{equation} \label{eq-LFnullo}
	\begin{array}{l}
\displaystyle	L_1(\theta)=\gamma_1 n\phi_1(+\infty) (\Lambda_1 (\theta_{1}-\theta_{2})-\alpha_1)\\
\\
\displaystyle L_2(\theta)=-\dfrac{\gamma_2}{2} \pi \mu c_{r,2} \cos r \theta_2+\gamma_2 n\phi_2(+\infty) (\Lambda_2 (\theta_{2}-\theta_{1})-\alpha_2),
	\end{array}
	\end{equation}
	for every $\theta \in \RR$.
	
	\noindent
	Now, let us look for solutions of $L(\theta)=0$; from the relation $L_1(\theta)=0$, recalling that $\phi_1(+\infty)\neq 0$, we deduce 
	\begin{equation} \label{eq-proofoggi11}
	\Lambda_1 (\theta_{1}-\theta_{2})=\alpha_1.
	\end{equation}
	From Lemma~\ref{lem-propLambda}, taking into account \eqref{eq-immaginelambda1}, we infer that there exists $\Lambda_1^* \in (0,\pi)$ such that 
	\begin{equation} \label{eq-risollambda1}
	\begin{array}{l}
	\Lambda_1(t)=\alpha_1 \quad \Longleftrightarrow \quad t=\pm \Lambda_1^*+2m\pi,\ m\in \mathbb{Z}\\
	\\
	\textnormal{sgn}(\Lambda'_1(\pm \Lambda_1^*))=\mp 1.
	\end{array}
	\end{equation}
	In particular, we choose $m=0$; then, from \eqref{eq-proofoggi11} and \eqref{eq-risollambda1} we obtain
	\begin{equation} \label{eq-solL1}
	\theta_1-\theta_2=\pm \Lambda_1^*.
	\end{equation}
	Replacing the last equality in the expression of $L_2$ in \eqref{eq-LFnullo} and recalling that $\Lambda_2$ is even and $2\pi$-periodic, the equation $L_2(\theta)=0$ reduces to
	\begin{equation} \label{eq-proofoggi12}
	-\pi \mu c_{r,2} \cos r \theta_2+2 n\phi_2(+\infty) (\Lambda_2 (\Lambda_1^*)-\alpha_2)=0,
	\end{equation}
	i.e. 
	\begin{equation} \label{eq-proofoggi13}
	\cos r \theta_2 = \dfrac{2 n\phi_2(+\infty) (\Lambda_2 (\Lambda_1^*)-\alpha_2)}{\pi \mu c_{r,2}}.
	\end{equation}
	Let now set
	\begin{equation} \label{eq-proofoggi14}
	\hat{\mu}= \left|\dfrac{2 n\phi_2(+\infty) (\Lambda_2 (\Lambda_1^*)-\alpha_2)}{\pi c_{r,2}}\right|;
	\end{equation}
	then, for every $\mu>\hat{\mu}$ the equation \eqref{eq-proofoggi13} can be solved and we obtain
	\begin{equation} \label{eq-proofoggi15}
	\theta_2 = \dfrac{1}{r} \left(\pm \arccos \dfrac{2 n\phi_2(+\infty) (\Lambda_2 (\Lambda_1^*)-\alpha_2)}{\pi \mu c_{r,2}} + 2h\pi\right),\quad h\in \mathbb{Z}.
	\end{equation}
	Choosing $h = 0$, we then conclude that, for every $\mu >\hat{\mu}$, the equation $L(\theta)=0$ has the four solutions
	$$
	\omega^{\pm,1}_{\mu} = \left(
		\Lambda_1^* + \dfrac{1}{r} \left(\pm \arccos \dfrac{2 n\phi_2(+\infty) (\Lambda_2 (\Lambda_1^*)-\alpha_2)}{\pi \mu c_{r,2}}\right),
	\dfrac{1}{r} \left(\pm \arccos \dfrac{2 n\phi_2(+\infty) (\Lambda_2 (\Lambda_1^*)-\alpha_2)}{\pi \mu c_{r,2}} \right)	
	\right)
	$$
	and
	$$
	\omega^{\pm,2}_{\mu} = \left(
	-\Lambda_1^* +\dfrac{1}{r} \left(\pm \arccos \dfrac{2 n\phi_2(+\infty) (\Lambda_2 (\Lambda_1^*)-\alpha_2)}{\pi \mu c_{r,2}} \right),
	\dfrac{1}{r} \left(\pm \arccos \dfrac{2 n\phi_2(+\infty) (\Lambda_2 (\Lambda_1^*)-\alpha_2)}{\pi \mu c_{r,2}}\right)
	\right).
	$$
	In order to apply Theorem \ref{teo-main}, we claim that one of the above four solutions, to be named $\omega^+$, is such that $JL(\omega^+)$ is a $\mathcal{D}^+$-matrix and another one, to be named $\omega^-$, is such that $JL(\omega^-)$ is a $\mathcal{D}^-$-matrix. To do this, recalling \eqref{eq-LFnullo} and the fact that $\Lambda'_i=\Sigma_i/n$ is $2\pi$-periodic and odd, we observe that 
	\begin{equation} 
	\begin{array}{ll}
	\displaystyle \partial_1 L_1 (\omega^{\pm,i}_{\mu})=&\displaystyle (-1)^{i+1} \gamma_1 \phi_1(+\infty) \Sigma_1 (\Lambda_1^*)\\
	&\\
	\displaystyle \partial_2 L_1 (\omega^{\pm,i}_{\mu})=&\displaystyle (-1)^{i} \gamma_1 \phi_1(+\infty) \Sigma_1 (\Lambda_1^*)\\
	&\\
	\displaystyle \partial_1 L_2 (\omega^{\pm,i}_{\mu})=&\displaystyle (-1)^{i} \gamma_2 \phi_2(+\infty) \Sigma_2 (\Lambda_1^*)\\
	&\\
	\displaystyle \partial_2 L_2 (\omega^{\pm,i}_{\mu})=&\displaystyle \pm \dfrac{\gamma_2}{2} \pi \mu c_{r,2} r \sin \arccos \dfrac{2 n\phi_2(+\infty) (\Lambda_2 (\Lambda_1^*)-\alpha_2)}{\pi \mu c_{r,2}} +\displaystyle (-1)^{i+1} \gamma_2 \phi_2(+\infty) \Sigma_2 (\Lambda_1^*),
	\end{array}
	\end{equation}
	for $i=1,2$. Now, since $\phi_1(+\infty)\neq 0$ and recalling \eqref{eq-risollambda1}, we have
	\begin{equation} \label{eq-proofoggi19}
	\operatorname{sgn} (\partial_1 L_{1} (\omega^{\pm,i}_{\mu})) = (-1)^{i} \operatorname{sgn} (\phi_1(+\infty));
	\end{equation}
	moreover, there exists $\check{\mu}\geq \hat{\mu}$ such that for every $\mu >\check{\mu}$ we have
	\begin{equation} \label{eq-proofoggi20}
	\operatorname{sgn} (\partial_2 L_{2} (\omega^{\pm,i}_{\mu})) = \pm \operatorname{sgn} (c_{r,2}).
	\end{equation}
	Hence, for $\mu >\check{\mu}$, the choice of $\omega^{\pm,i}_{\mu}$ has to be made according to the signs of $\phi_1(+\infty)$ and $c_{r,2}$.

	
	\noindent
	For the sake of briefness, we discuss the case $\phi_1(+\infty)>0$ and $c_{r,2}>0$, the other ones being similar. We set $\omega^+= \omega^{-,1}_{\mu}$ and $\omega^-= \omega^{+,2}_{\mu}$; hence, by construction,
	 $JL(\omega^+)$ and $JL (\omega^-)$ satisfy the sign conditions on the diagonal coefficients in order to be a $\mathcal{D}^\pm$-matrix. As far as the third condition in Definition \ref{def-dpiumeno} is concerned, we have that
	\begin{equation} \label{eq-proofoggi22}
	\begin{array}{l}
	\displaystyle \partial_1 L_{1} (\omega^\pm)\, \partial_1 L_{2} (\omega^\pm)+\partial_2 L_{1} (\omega^\pm)\, \partial_2 L_{2} (\omega^\pm)=\displaystyle -2 \gamma_1 \gamma_2 \phi_1(+\infty) \phi_2(+\infty) \Sigma_1 (\Lambda_1^*) \Sigma_2 (\Lambda_1^*)\\
	\\
	\qquad \displaystyle  + \gamma_1 \dfrac{\gamma_2}{2} \pi \mu c_{r,2} rn \sin \arccos \dfrac{2 n\phi_2(+\infty) (\Lambda_2 (\Lambda_1^*)-\alpha_2)}{\pi \mu c_{r,2}} \phi_1(+\infty) \Sigma_1 (\Lambda_1^*)
	\end{array}
	\end{equation}
	and
	\begin{equation} \label{eq-proofoggi23}
	\begin{array}{l}
	\displaystyle 2 \partial_1 L_{1} (\omega^\pm)\, \partial_2 L_{2} (\omega^\pm)=\displaystyle 2 \gamma_1 \gamma_2 \phi_1(+\infty) \phi_2(+\infty) \Sigma_1 (\Lambda_1^*) \Sigma_2 (\Lambda_1^*)+\\
	\\
	\qquad \displaystyle - \gamma_1 \gamma_2 \pi \mu c_{r,2} rn \sin \arccos \dfrac{2 n\phi_2(+\infty) (\Lambda_2 (\Lambda_1^*)-\alpha_2)}{\pi \mu c_{r,2}} \phi_1(+\infty) \Sigma_1 (\Lambda_1^*).
	\end{array}
	\end{equation}
	Hence, there exists $\mu^*\geq \check{\mu}$ such that for every $\mu >\mu^*$ the third condition in Definition \ref{def-dpiumeno} is satisfied; hence the values $\omega^\pm$ are such that $JL(\omega^\pm)$ is a $\mathcal{D}^\pm$ matrix. The thesis then follows from an application of Theorem \ref{teo-main}.
\end{proof}

As a last application, we discuss the case when the oscillators are symmetric, i.e. $a_i=b_i = n^2$ for $i=1,2$, and \eqref{eq-coeffFnullo} holds true; as already observed, this is equivalent to the assumption
\begin{equation}\label{eq-condp1}
\widehat{p}_{1,n}=0,
\end{equation}
where $\widehat{p}_{1,n}$ is as in \eqref{eq-fourier}. 
Let us observe that this situation is not covered by the results in \cite{BosDamPapPP}.

\begin{corollary} \label{cor-noJo}
Let $a_i=b_i=n^2$, for $i=1,2$, and suppose that conditions \eqref{eq-rellimitiphii2} and \eqref{eq-condp1} are satisfied. 

\noindent
Then, for every $\phi_1(+\infty)\neq 0$ and for every $\phi_2(+\infty)\in \R$ such that
\begin{equation} \label{eq-condnoJo}
|\phi_2(+\infty)| < \dfrac{3}{16} |\widehat{p}_{2,n}|,
\end{equation}
with ${\widehat p}_{2,n}$ as in \eqref{eq-fourier}, there exist two infinite measure sets $E^\pm\subset \RR\times \RR$ such that:
	\begin{itemize}
	\item for every solution $x$ of \eqref{oscacc} such that $(x(0),x'(0))\in E^+$, 
	\[
	\lim_{t\to +\infty} (|x_i(t)|^2+|x'_i(t)|^2) =+\infty, \qquad i=1,2,
	\]
	\item for every solution $x$ of \eqref{oscacc} such that $(x(0),x'(0))\in E^-$
	\[
	\lim_{t\to -\infty} (|x_i(t)|^2+|x'_i(t)|^2) =+\infty, \qquad i=1,2.
	\] 
	\end{itemize}
\end{corollary}

\begin{proof} 
	First of all, let us observe that in this situation the functions $C_i$ and $\Lambda_i$ in \eqref{eq-defCi} and \eqref{eq-deflambda} are given by
	\begin{equation} \label{eq-pr100}
	C_i(t)=\cos n t,\quad \Lambda_i(t)=\dfrac{2}{n}\cos t,\quad \forall \ t\in \R,
	\end{equation}
	respectively, while the number $\alpha_i$ in \eqref{eq-defalphai} is zero. Moreover, the function $\Phi_{2}$ in \eqref{proof6} becomes
	 	\begin{equation} \label{eq-pr101}
	 \Phi_{2}(u)=-\dfrac{1}{\sqrt{2n}} \int_0^{2\pi} \cos (nt+u) p_{2}(t)dt,\quad \forall \ u\in \R;
	 \end{equation}
	this expression can be written as
		\begin{equation} \label{eq-pr102}
	\Phi_{2}(u)=-\dfrac{1}{\sqrt{2n}}\, |\widehat{p}_{2,n}| \cos (u+\psi_{2}),\quad \forall \ u\in \R,
	\end{equation}
	for some $\psi_2\in \R$. 
	
	\noindent
	From \eqref{eq-pr100}, \eqref{eq-pr101}, \eqref{eq-pr102} and the assumption on $\widehat{p}_{1,n}$ we deduce that
	\begin{equation} \label{eq-pr103}
	\begin{array}{ll}
\displaystyle L_1 (\theta)=2\sqrt{\frac{2}{n}} \phi_1(+\infty) \cos (\theta_1-\theta_2) \\
\\
\displaystyle L_2 (\theta)=-\dfrac{1}{\sqrt{2n}}\, |\widehat{p}_{2,n}| \cos (\theta_2+\psi_2)+2\sqrt{\frac{2}{n}} \phi_2(+\infty) \cos (\theta_2-\theta_1),
	\end{array}
	\end{equation}
	for every $\theta \in \RR$.
	
	\noindent
	Recalling that $\phi_1(+\infty)\neq 0$, 
	we can solve the equation $L_1(\theta)=0$ to obtain
	\begin{equation} \label{eq-pr104}
	\theta_1=\theta_2 \pm \dfrac{\pi}{2} +2m\pi,\quad m\in \mathbb{Z};
	\end{equation}
	as a consequence the equation $L_2 (\theta)=0$ reduces to
	\begin{equation} \label{eq-pr105}
|\widehat{p}_{2,n}| \cos (\theta_2+\psi_2)=0.
	\end{equation}
	We now observe that assumption \eqref{eq-condnoJo} implies that $\widehat{p}_{2,n}\neq 0$; hence, from \eqref{eq-pr105} we infer that 
	\begin{equation} \label{eq-pr106}
\theta_2=-\psi_2 \pm \dfrac{\pi}{2} +2h\pi,\quad h\in \mathbb{Z}.
	\end{equation}
	Choosing in particular $m = h = 0$, we conclude that the equation $L(\theta)=0$ has the four solutions 
	\begin{equation} \label{eq-pr107}
	\begin{array}{l}
	\displaystyle \omega^{\pm,1} = \left(-\psi_2 +\dfrac{\pi}{2} \pm \dfrac{\pi}{2} , -\psi_2 \pm \dfrac{\pi}{2} \right) \in \mathbb{R}^2,  \\
	\\
	\displaystyle \omega^{\pm,2} = \left(-\psi_2 -\dfrac{\pi}{2} \pm \dfrac{\pi}{2} , -\psi_2 \pm \dfrac{\pi}{2} \right) \in \mathbb{R}^2.
	\end{array}
	\end{equation}
	We now claim that one of the above four solutions, to be named $\omega^+$, is such that $JL(\omega^+)$ is a $\mathcal{D}^+$-matrix and another one, to be named $\omega^-$, is such that $JL(\omega^-)$ is a $\mathcal{D}^-$-matrix.
	
	\noindent
	To see this, let us observe that, from \eqref{eq-pr107},
	\begin{equation} \label{eq-pr109}
	\begin{array}{ll}
	\displaystyle \partial_1 L_1 (\omega^{\pm,i})&=\displaystyle \partial_1 L_1 (\theta)_{\vert \theta = \omega^{\pm,i}}=\displaystyle -2\sqrt{\frac{2}{n}} \phi_1(+\infty) \sin (\theta_1-\theta_2)_{\vert \theta = \omega^{\pm,i}}=\displaystyle (-1)^{i} 2\sqrt{\frac{2}{n}} \phi_1(+\infty) \\
	\\
	\displaystyle \partial_2 L_1 (\omega^{\pm,i})&=\displaystyle  \partial_2 L_1 (\theta)_{\vert \theta = \omega^{\pm,i}} =2\sqrt{\frac{2}{n}} \phi_1(+\infty) \sin (\theta_1-\theta_2)_{\vert \theta = \omega^{\pm,i}} = (-1)^{i-1} 2\sqrt{\frac{2}{n}} \phi_1(+\infty) \\
	\\
	\displaystyle \partial_1 L_2 (\omega^{\pm,i})&=\displaystyle  \partial_1 L_2 (\theta)_{\vert \theta = \omega^{\pm,i}} = 
	2\sqrt{\frac{2}{n}} \phi_2(+\infty) \sin (\theta_2-\theta_1)_{\vert \theta = \omega^{\pm,i}} = (-1)^{i} 2\sqrt{\frac{2}{n}} \phi_2(+\infty)   \\
	\\
	\displaystyle \partial_2 L_2 (\omega^{\pm,i})&=\displaystyle  \partial_2 L_2 (\theta)_{\vert \theta = \omega^{\pm,i}} = \dfrac{1}{\sqrt{2n}} |\widehat{p}_{2,n}| \sin (\theta_2+\psi_2)
	-2\sqrt{\frac{2}{n}} \phi_2(+\infty) \sin (\theta_2-\theta_1)_{\vert \theta = \omega^{\pm,i}} \\
	&=  \displaystyle \pm \dfrac{1}{\sqrt{2n}} |\widehat{p}_{2,n}|
	+ (-1)^{i-1} 2\sqrt{\frac{2}{n}} \phi_2(+\infty),
	\end{array}
	\end{equation}
	for $i=1,2$. 
	Focusing for the sake of briefness on the case $\phi_1(+\infty)>0$, we obtain from \eqref{eq-condnoJo} that
	\begin{equation} \label{eq-pr110}
	\begin{array}{l}
	\operatorname{sgn} (\partial_1 L_{1} (\omega^{-,1}))\cdot \operatorname{sgn} (\partial_2 L_{2} (\omega^{-,1})) >0\\
	\\
	\operatorname{sgn} (\partial_1 L_{1} (\omega^{+,2}))\cdot \operatorname{sgn} (\partial_2 L_{2} (\omega^{+,2})) >0.
	\end{array}
	\end{equation}
	Setting $\omega^+ = \omega^{+,2}$ and $\omega^-=\omega^{-,1}$, since 
	\begin{equation} \label{eq-pr111}
	\begin{array}{l}
	\displaystyle \partial_1 L_{1} (\omega^\pm)\, \partial_1 L_{2} (\omega^\pm)+\partial_2 L_{1} (\omega^\pm)\, \partial_2 L_{2} (\omega^\pm)=\displaystyle \frac{1}{n}\left(16\phi_1(+\infty)\phi_2(+\infty) -2|\widehat{p}_{2,n}|\phi_1(+\infty)\right)\\
	\\
	\displaystyle  
	\end{array}
	\end{equation}
	and
	\begin{equation} \label{eq-pr112}
	\begin{array}{l}
	\displaystyle 2 \partial_1 L_{1} (\omega^\pm)\, \partial_2 L_{2} (\omega^\pm)=\frac{1}{n}\left(-16\phi_1(+\infty)\phi_2(+\infty) +4|\widehat{p}_{2,n}|\phi_1(+\infty)\right),
	\end{array}
	\end{equation}
	from the same assumption \eqref{eq-condnoJo} we deduce that
	\begin{equation} \label{eq-pr113}
	\begin{array}{l}
	|\partial_1 L_{1} (\omega^\pm)\, \partial_1 L_{2} (\omega^\pm)+\partial_2 L_{1} (\omega^\pm)\, \partial_2 L_{2} (\omega^\pm)|<2 \partial_1 L_{1} (\omega^\pm)\, \partial_2 L_{2} (\omega^\pm)
	\end{array}
	\end{equation}
	as well. From \eqref{eq-pr110} and \eqref{eq-pr113} we conclude that $JL(\omega^\pm)$ is a $\mathcal{D}^\pm$-matrix.
	The thesis then follows from an application of Theorem \ref{teo-main}.
\end{proof}

\medbreak
\noindent
\textbf{Acknowledgments.} The authors are grateful to Rafael Ortega for having proposed the subject of this investigation and for his enduring encouragement.

\end{document}